\documentclass[11.5pt]{article}
\usepackage{textcomp}
\usepackage{amsmath,amsfonts}
\usepackage{caption}
\usepackage{accents}
\usepackage{multirow}

\usepackage{lineno,hyperref}
\modulolinenumbers[5]

\usepackage[T1]{fontenc}
\usepackage{color}
\usepackage{fancyhdr}
\usepackage{fullpage}
\usepackage{setspace}
\doublespace
\usepackage{comment}
\usepackage{makeidx}
\usepackage{hyperref}
\usepackage{cleveref}
\crefformat{footnote}{#2\footnotemark[#1]#3}

\usepackage[toc,page]{appendix}
\usepackage{multicol}
\usepackage{amsmath, amsfonts, amssymb}
\usepackage{amsthm}
\usepackage{graphicx}
\usepackage{pst-all,pst-infixplot,pst-math,pst-fractal,pst-solides3d}

\newtheorem{teo}{Theorem}[section]
\newtheorem{prop}{Proposition}[section]

\newtheorem{lemma}{Lemma}[section]

\newtheorem{assumption}{Assumption}[section]

\theoremstyle{definition}
\newtheorem{defi}[teo]{Definition}

\theoremstyle{remark}

\begin{document}

\author{Manuel Guerra\thanks{CEMAPRE, Instituto Superior de Economia e Gest\~ao, Universidade de Lisboa, Rua do Quelhas 6, Lisbon, Portugal, Email: mguerra@iseg.ulisboa.pt}, \
Cl\'audia Nunes\thanks{Department of Mathematics and
CEMAT, Instituto Superior T\'ecnico, Universidade de
Lisboa, Av. Rovisco Pais, 1049-001 Lisboa, Portugal,
Email: cnunes@math.tecnico.ulisboa.pt}
\ and Carlos Oliveira\thanks{Department of Mathematics and
CEMAT, Instituto Superior T\'ecnico, Universidade de
Lisboa, Av. Rovisco Pais, 1049-001 Lisboa, Portugal,
Email: carlosmoliveira@tecnico.ulisboa.pt}}

\title{\sc Optimal stopping of one-dimensional diffusions with integral criteria}

\maketitle

\begin{abstract}

This paper provides a full characterization of the value function and solution(s) of an optimal stopping problem for a one-dimensional diffusion with an integral criterion.
The results hold under very weak assumptions, namely,
the diffusion is assumed to be a weak solution of stochastic differential equation satisfying the Engelbert-Schmidt conditions, while the (stochastic) discount rate and the integrand are required to satisfy only general integrability conditions.

\mbox{}


\mbox{}
\newline
{\bf AMS (2010) Subject Classifications.\/} Primary 60G40;
secondary 60H10, 93E20.
\end{abstract}

%
%

\section{Introduction}
\label{S introduction}

Optimal stopping problems attracted generations of mathematicians due to both their interesting mathematical characteristics and their important applications.  
Early work was developped by Dynkin \cite{dynkin1963optimal}, Grigelionis and Shiryaev \cite{grigelionis1966stefan}, Dynkin and Yushkevich \cite{DynkinYushkevich69}.
A general theory can be found in books by Shiryaev \cite{Shiryaev1978} and Peskir and Shiryaev \cite{Shiryaev}.
Several methods have been developed to deal with this type of problems.

Methods based on excessive functions date back to the pioneer work of 
Dynkin \cite{dynkin1963optimal}, 
and have been used by, among others, 
Dynkin and Yushkevich \cite{DynkinYushkevich69}, 
Fakeev \cite{Fakeev71}, 
Thompson \cite{Thompson71}, 
Shiryaev \cite{Shiryaev1978}, 
Salminen \cite{salminen1983optimal}, 
Alvarez \cite{Alvarez03}, 
Dayanik and Karatzas \cite{dayanik2003optimal}, 
Lamberton and Zervos \cite{lamberton2013optimal}, 
among others.
These methods are tightly connected with the concavity and monotonicity properties of the value function.

An alternative approach 
based on variational methods and inequalities was pioneered by
Grigelionis and Shiryaev \cite{grigelionis1966stefan},
and
Bensoussan and Lions \cite{BensoussanLions73}.
It was used in many works, namely 
Nagai \cite{nagai1978optimal}, 
Friedman \cite{Friedman76},
Krylov \cite{krylov2008controlled}, 
Bensoussan and Lions \cite{bensoussan2011applications}
\O ksendal \cite{Oksendal},
Lamberton \cite{Lamberton09}, 
Lamberton and Zervos \cite{lamberton2013optimal}, 
R\"uschendorf and Urusov \cite{ruschendorf2008class},
Belomestny, R\"uschendorf and Urusov \cite{belomestny2010optimal},
among others.
Usually this approach requires some regularity assumptions on the problem's data and on the value function.
Progress has been made in relaxing these assumptions, showing that the value function satisfies the appropriate variational inequality in various weak senses (see, for example,
Friedman \cite{Friedman76},
Nagai \cite{nagai1978optimal}, 
Zabczyk \cite{Zabczyk84},
\O ksendal and Reikvam \cite{OksendalReikvam98},
Bassan and Ceci \cite{BassanCeci02},
Bensoussan and Lions \cite{bensoussan2011applications},
Lamberton \cite{Lamberton09}, 
Lamberton and Zervos \cite{lamberton2013optimal}).
The variational approach allows for the development of some effective numerical methods 
(see, for example,
Glowinski, Lions and Tr{\'e}moli{\`e}res \cite{glowinski1981numerical},
or
Zhang \cite{Zhang94}).

A third approach, based on change of measure techniques and martingale theory, was introduced by 
Beibel and Lerche \cite{BeibelLerche97,BeibelLerche02},
and was further developed by several authors, namely
Alvarez \cite{Alvarez03,Alvarez04,Alvarez08},
Lerche and Urusov \cite{LercheUrusov07},
Lempa \cite{Lempa10},
Christensen and Irle \cite{ChristensenIrle11}.
This approach proved successful in characterizing the optimal strategy at any given point of the state space.

In this paper we consider the optimal stopping problem of a general diffusion when the optimality criterion is an integral functional.
More precisely, we seek the stopping time $\hat \tau$ maximizing the expected outcome
\begin{equation}
\label{Eq outcome}
J(x,\tau) = \mathbb E_x \left[ \int_0^\tau e^{-\rho_s} \Pi(X_s) ds \right] ,
\end{equation}
where 
\begin{equation}
\label{Eq discount factor}
\rho_t=\int_0^t r(X_s) ds \qquad \forall t \geq 0, 
\end{equation}
and $X$ solves the stochastic differential equation
\begin{equation}
\label{Eq SDE}
dX_t = \alpha (X_t) dt + \sigma (X_t) dW_t .
\end{equation}
$\mathbb E_x$ means expected value conditional on $X_0=x$,
$W$ is a standard Brownian motion and $r$, $\alpha $, $\sigma $ and $\Pi$ are measurable real functions, satisfying minimal assumptions discussed in Section \ref{S problem setting} below. 
In particular, the functions $r$, $\alpha $, $\sigma $ and $\Pi$ may be discontinuous.
As usual, $\tau$ is an admissible stopping time if and only if it is a stopping time with respect to the filtration generated by the process $X$. 

This class of optimal stopping problems has received  little attention compared with optimal stopping problems where the functional being maximized is of type
\begin{equation}
\label{Eq terminal criterion}
\tilde J(x, \tau ) = \mathbb E_x \left[ e^{-\rho_\tau} \Pi (X_\tau) \chi_{\tau < +\infty} \right] .
\end{equation}
This is understandable, since the functional \eqref{Eq terminal criterion} arises naturally in many applications, particularly in the theory of American Options in mathematical finance.
However, the problem \eqref{Eq outcome}--\eqref{Eq discount factor}--\eqref{Eq SDE} also has important applications, among others, in the theories of Asian Options and Real Options.
Further, some known problems in the literature of optimal stopping and stochastic control can be reduced to the form \eqref{Eq outcome}--\eqref{Eq discount factor}--\eqref{Eq SDE} (see for example, Graversen, Peskir and Shiryaev \cite{graversen2001stopping}, and Karatzas and Ocone \cite{karatzas2002leavable}).

Our approach is closely related to the works of 
R{\"u}schendorf and Urusov \cite{ruschendorf2008class},
and
Belomestny, R{\"u}schendorf and Urusov \cite{belomestny2010optimal}.
We show that the value function solves a variational inequality in the Carath\'eodory sense.
Thus, it is a continuously differentiable function with absolutely continuous first derivative, and it is not necessary to consider further weak solutions.
The free boundary is fixed by a \textit{$C_1$ fit condition}, coupled with a \textit{global} non-negativity condition.
Notice that the necessity (or not) of a \textit{smooth fit principle} is a topic of current literature. 
For instance, works by 
Dayanik and Karatzas \cite{dayanik2003optimal} (section 7), 
Villeneuve \cite{villeneuve2007threshold},
R{\"u}schendorf and Urusov \cite{ruschendorf2008class}, 
Belomestny, R{\"u}schendorf and Urusov \cite{belomestny2010optimal}, 
and 
Lamberton and Zervos \cite{lamberton2013optimal}, 
prove that in certain cases, the \textit{smooth fit principle} holds. 
This contrasts with works by 
Salminen \cite{salminen1983optimal}, 
Peskir \cite{peskir2007principle}, 
and 
Samee \cite{samee2010principle}, 
which find examples where the \textit{smooth fit principle} fails.

R{\"u}schendorf and Urusov \cite{ruschendorf2008class} 
and
Belomestny, R{\"u}schendorf and Urusov \cite{belomestny2010optimal} 
deal with the problem \eqref{Eq outcome}--\eqref{Eq discount factor}--\eqref{Eq SDE} assuming that the function $\Pi$ is of so-called ``two-sided form''.
The corresponding variational inequality is solved assuming a priori that the value function coincides on its support with the solution of an ordinary differential equation with two-sided zero boundary condition.
Therefore, the method does not provide any information in cases when the value function is of some other form (e.g., a solution of the differential equation with only one-sided zero boundary condition), even if $\Pi$ belongs to the restricted class of functions of ``two-sided form''.
In this paper, we solve the variational inequality without assuming any particular behaviour for $\Pi$ or the value function, obtaining a characterization of the value function in terms of $\Pi$ and the fundamental solution of a system of linear differential equations.
As can be expected with this generality, the value function can assume many different forms, but it can always be found, at least on a given compact interval, by solving a finite-dimensional system of nonlinear equations.
In particular, we address the issues raised in the remarks after Theorem 2.2 and in the remarks after Theorem 2.3 of R{\"u}schendorf and Urusov \cite{ruschendorf2008class}, as well as in the remarks after Theorem 2.2 of Belomestny, R{\"u}schendorf and Urusov \cite{belomestny2010optimal}.

Lamberton and Zervos \cite{lamberton2013optimal} show that the value function for the problem \eqref{Eq terminal criterion}--\eqref{Eq discount factor}--\eqref{Eq SDE}  is the difference between two convex functions.
Every function with absolutely continuous first derivative can be represented as the difference between two convex functions, but the converse is not true, since the derivative of a convex function can have countably many points of discontinuity.
Thus our results show that the value function for the problem \eqref{Eq outcome}--\eqref{Eq discount factor}--\eqref{Eq SDE} is somewhat more regular than the solutions in \cite{lamberton2013optimal}.
On the other hand,
Dayanik and Karatzas \cite{dayanik2003optimal} 
proved that the value function of \eqref{Eq terminal criterion}--\eqref{Eq discount factor}--\eqref{Eq SDE} is concave with respect to the scale function of the process $X$.
We show that this result does not extend to the problem \eqref{Eq outcome}--\eqref{Eq discount factor}--\eqref{Eq SDE}, providing an example where the value function does not admit any strictly increasing function $F$ with respect to which it is $F$-concave.

This paper is organized as follows.
Section \ref{S problem setting} contains the complete definition of problem \eqref{Eq outcome}--\eqref{Eq discount factor}--\eqref{Eq SDE}, with the formulation of our working assumptions.
Section \ref{S background} contains an outline of some elementary background material and sets some notation not introduced in Section \ref{S problem setting}.
Section \ref{S results} contains the main results in the paper and some discussion on their usage to solve problems of type \eqref{Eq outcome}--\eqref{Eq discount factor}--\eqref{Eq SDE}.
Proofs of these results are postponed to Section \ref{S proofs}.
Section \ref{S example} contains some examples of solutions of optimal stopping problems.

\section{Problem setting}\label{S problem setting}

Let $\alpha , r , \Pi: I \mapsto \mathbb R$, $\sigma : I \mapsto ]0,+\infty[$ be Borel-measurable functions, where $I = ] m,M[$ is an open interval with $-\infty \leq m < M \leq + \infty$.
$\overline I = I \cup \{ \infty \}$ denotes  the one-point (Aleksandrov) compactification of $I$.

\begin{assumption}
\label{Assumption Engelbert}
The functions $\frac{1}{\sigma^2}$, $\frac{\alpha}{\sigma ^2}$ are locally integrable with respect to the Lebesgue measure in $I$. 
\end{assumption}

By Theorem 5.15 in chapter 5 of \cite{karatzas2012brownian}, Assumption \ref{Assumption Engelbert} guarantees existence and uniqueness (in law) of a weak solution for the stochastic differential equation \eqref{Eq SDE}, up to explosion time.
In all the following, $(\Omega, \mathcal F, \{ \mathcal F_t \}_{t \geq 0} , P, X, W)$ denotes a given weak solution up to explosion time of equation \eqref{Eq SDE}.
$\tau_I$ denotes the explosion time, and the process $X$ is extended to the time interval $[0,+\infty[$ by setting $X_t=\infty$ for $t \geq \tau_I$.
We extend the functions $r$, $\Pi$ into $\overline{ I}$, setting $r(\infty ) = 0$, $\Pi(\infty ) = - \infty $.
Thus, the processes $r(X_t)$, $\Pi(X_t)$ are well defined on the time interval $[0,+\infty[$.

For every $t \geq 0$, $\mathcal F_t^X$ is the $\sigma$-algebra generated by $\{ X_s\}_{0 \leq s \leq t}$, augmented with all the $P$-null events.
The set of admissible stopping times for expression \eqref{Eq outcome}, denoted by $\mathcal T$, is the set of all stopping times adapted to the filtration $\{ \mathcal F^X_t \}_{t \geq 0}$.

For any real-valued function $f$, we set 
\begin{equation}
\notag
f^+(x) = \max (f(x), 0),
\qquad
f^-(x) = \max (-f(x), 0) .
\end{equation}

Besides Assumption \ref{Assumption Engelbert}, we take the following assumptions concerning the functional \eqref{Eq outcome}:

\begin{assumption}
\label{Assumption r}
The function $\frac{r}{\sigma^2}$ is locally integrable with respect to the Lebesgue measure in $I$. 
\end{assumption}

\begin{assumption}
\label{Assumption Pi}
The function $\frac{\Pi}{\sigma^2}$ is locally integrable with respect to the Lebesgue measure in $I$,
the sets $\left\{ x \in I:\Pi(x) >0 \right\}$ and $\left\{ x \in I:\Pi(x) <0 \right\}$ have both positive Lebesgue measure, and
\begin{equation}
\notag 
\mathbb E_x \left[ \int_0^{\tau_I} e^{-\rho_t} \Pi^+ ( X_t) dt \right] < + \infty
\qquad \forall x \in I .
\end{equation}
\end{assumption}

It turns out (see Proposition \ref{P assumption Pi vs Pia}) that Assumption \ref{Assumption Pi} is equivalent to the apparently weaker:
\begin{assumption}
\label{Assumption Pia}
The function $\frac{\Pi}{\sigma^2}$ is locally integrable with respect to the Lebesgue measure in $I$,
the sets $\left\{ x \in I:\Pi(x) >0 \right\}$ and $\left\{ x \in I:\Pi(x) <0 \right\}$ have both positive Lebesgue measure, and there is some $x \in I$ such that
\begin{equation}
\notag
\mathbb E_x \left[ \int_0^{\tau_I} e^{-\rho_t} \Pi^+ ( X_t) dt \right] < + \infty .
\end{equation}
\end{assumption}

We will see in Section \ref{S background} that local integrability of $\frac \alpha{\sigma^2}$, $\frac r{\sigma^2}$ and $\frac \Pi{\sigma^2}$ is necessary and sufficient for existence of solution for the equation \eqref{Eq ODE} and therefore, it is necessary for existence of solution of the variational inequality \eqref{Eq HJB}.
Further, if the set $\left\{ x \in I:\Pi(x) >0 \right\}$ is negligible, then $\tau \equiv 0$ is trivially optimal.
Conversely, when the set $\left\{ x \in I:\Pi(x) <0 \right\}$ is negligible, then $\tau_I$ is trivially optimal.
Taking into account the equivalence between Assumptions \ref{Assumption Pi} and \ref{Assumption Pia}, 
if $\mathbb E_x \left[ \int_0^{\tau_I} e^{-\rho_t} \Pi^+ ( X_t) dt \right] = + \infty $ and $\mathbb E_x \left[ \int_0^{\tau_I} e^{-\rho_t} \Pi^- ( X_t) dt \right] < + \infty $ then $\tau_I$ is trivially optimal.
If $\mathbb E_x \left[ \int_0^{\tau_I} e^{-\rho_t} \Pi^+ ( X_t) dt \right] = \mathbb E_x \left[ \int_0^{\tau_I} e^{-\rho_t} \Pi^- ( X_t) dt \right] = + \infty $ then, the functional \eqref{Eq outcome} is not well defined at least for some stopping times $\tau \in \mathcal T$.
Thus, the Assumption \ref{Assumption Pi}/\ref{Assumption Pia} excludes some trivial cases and some ill-posed cases.

The optimal stopping problem considered in this paper consists of finding the maximizers of \eqref{Eq outcome} over the set $\mathcal T$. This is equivalent to finding the \emph{value function}
\begin{equation}
\label{Eq value function}
V(x) = \sup_{\tau \in \mathcal{T}} \mathbb E_x \left[ \int_0^\tau e^{-\rho_s} \Pi (X_s) ds \right] .
\end{equation}
Since the strategy $\tau \equiv 0$ (to stop immediately, regardless of the current state $X_0$) has zero payoff, it is obvious that $V$ is a non-negative function.
An optimal stopping time is given by the rule
\begin{equation}
\notag
\tau = \inf \left\{ t\geq 0 : V(X_t) = 0 \right\} .
\end{equation}

\section{Background and notation}\label{S background}

Taking into account the general results relating variational inequalities with optimal stopping (see, e.g. Peskir and Shiriaev \cite{Shiryaev} or Krylov \cite{krylov2008controlled}), it is expected that the value function \eqref{Eq value function} satisfies the Hamilton-Jacobi-Bellman equation
\begin{equation}\label{Eq HJB}
\min\left\{r(x) v(x)-\alpha(x) v^\prime (x)-\frac{\sigma(x)^2}{2} v^{\prime \prime} (x)-\Pi(x), v(x)\right\}=0 .
\end{equation}
Often, similar variational inequalities are presented in  slightly different forms, as free boundary problems, as in Grigelionis and Shiryaev \cite{grigelionis1966stefan}.
Obviously any solution $v$ of \eqref{Eq HJB} must coincide with a solution of the ordinary differential equation
\begin{equation}\label{Eq ODE}
r(x) v(x)-\alpha(x) v^\prime (x)-\frac{\sigma(x)^2}{2} v^{\prime \prime} (x)-\Pi(x)=0,
\end{equation}
in any interval where $v(x)>0$.
The equation \eqref{Eq ODE} is equivalent to the system of first-order differential equations
\begin{equation}\label{Eq system ODE}
w'(x)=A(x)w(x)+b(x),
\end{equation} 
where 
\begin{equation}
\notag
w(x)=\begin{pmatrix}
   v(x) \\
  v'(x)
 \end{pmatrix},\quad b(x)=\begin{pmatrix}
   0 \\
  -\frac{2\Pi (x)}{\sigma(x)^2}
 \end{pmatrix}\quad\text{and}\quad A(x)=\begin{pmatrix}
   0& 1 \\
  \frac{2r(x)}{\sigma(x)^2}& - \frac{2\alpha(x)}{\sigma(x)^2}
 \end{pmatrix}.
\end{equation}
Solutions for the system \eqref{Eq system ODE} are understood in the Carath\'eodory sense, that is, $w:I \mapsto \mathbb R^2$ is said to be a solution of \eqref{Eq system ODE} if it is absolutely continuous and satisfies
\begin{equation}
\notag
w(x) = w(a) + \int_a^x A(z) w(z) + b(z) dz 
\qquad 
\forall x \in I ,
\end{equation}
where $a$ is an arbitrary point of $I$.
Thus, the solutions of equation \eqref{Eq ODE} are continuously differentiable functions with absolutely continuous first derivatives.
Similarly, we say that a function $v$ is a solution of the Hamilton-Jacobi-Bellman equation \eqref{Eq HJB} if and only if $v$ is continuously differentiable, its first derivative is absolutely continuous, and $v$ satisfies \eqref{Eq HJB} almost everywhere with respect to the Lebesgue measure.
In other words, any solution $v$ of  equations \eqref{Eq HJB} or \eqref{Eq ODE} can be written as the difference between two convex functions with absolutely continuous derivatives.
This class of functions is a subset of the class used in \cite{lamberton2013optimal}, but we do not use this fact in this paper.

Let
\begin{equation}
\notag
\Phi (x,y)=\left( \begin{array}{cc}
\phi_{11}(x,y) & \phi_{12}(x,y) \\
\phi_{21}(x,y) & \phi_{22}(x,y)
\end{array} \right)
\end{equation}
be the fundamental solution of the homogeneous system $w^\prime = Aw$. 
That is, $\Phi$ the unique solution of the matrix differential equation
\begin{equation}
\notag
\frac{\partial}{\partial y}\Phi(x,y)=A(y)\Phi(x,y) ,
\qquad
\Phi(x,x) = Id 
\end{equation}
where $Id$ represents the identity matrix.

The Assumptions \ref{Assumption Engelbert} and \ref{Assumption r} are necessary and sufficient for existence of $\Phi(x,y)$ for every $x,y \in I$.
The additional Assumption \ref{Assumption Pi} guarantees existence of one unique solution for the non-homogeneous system \eqref{Eq system ODE} defined in the whole interval $I$, for every initial condition $v(a) = \hat v_1$, $v^\prime (a) = \hat v_2$ with $a \in I$, $\hat v_1, \hat v_2 \in \mathbb R$.
Any solution of \eqref{Eq system ODE} can be written in the form
\begin{equation}
\label{Eq solution system ODE}
w(x) = 
\Phi(a,x) \left( w(a) + \int_a^x \Phi(a,z)^{-1} b(z) dz \right) =
\Phi(a,x)w(a) + \int_a^x \Phi(z,x) b(z) dz ,
\end{equation}
where $a$ is an arbitrary point of $I$.
That is, any solution of \eqref{Eq ODE} can be written in the form
\begin{equation}
\label{Eq v nonhomogeneous} 
v(x) = 
v(a) \phi_{11}(a,x) + 
v^\prime (a)\phi_{12}(a,x) -
\int_a^{x}\frac{2\Pi(z)}{\sigma(z)^2}\phi_{12}(z,x)dz ,
\qquad
\forall x \in I .
\end{equation}
For any $a,b \in I$, with $a< b$, and any $d \in \mathbb R$, we introduce the functions
\begin{align}
& \label{Eq v_ad}
v_{a,d} (x) = d \phi_{12}(a,x) - \int_a^x \frac{2 \Pi (z)}{\sigma(z)^2} \phi_{12}(z,x) dz & x \in I,
\\ & \label{Eq vab}
v^{[a,b]}(x) = \frac{\int_a^b \frac{2 \Pi (z)}{\sigma(z)^2} \phi_{12}(z,b) dz}{\phi_{12}(a,b)} \phi_{12}(a,x) - \int_a^x \frac{2 \Pi (z)}{\sigma(z)^2} \phi_{12}(z,x) dz & x \in I.
\end{align}
These functions are, respectively, the solution of \eqref{Eq ODE} with initial conditions $v(a)=0$, $v^\prime (a) = d$, and the solution of \eqref{Eq ODE} with boundary conditions $v(a)=v(b)=0$.
We will show below (Proposition \ref{P sign phi_12}) that Assumption \ref{Assumption Pi} implies $\phi_{12}(a,b)>0$ for every $m<a<b<M$ and hence $v^{[a,b]}$ is well defined and is the unique solution of the corresponding boundary value problem. 
Belomestny, R{\"u}schendorf and Urusov \cite{belomestny2010optimal} proved a similar result using the probabilistic representation of such equation \eqref{Eq ODE}.
We provide a shorter and more general proof using classical arguments from the theory of ordinary differential equations.

If $a=m$ or $b=M$ (or both), then we can pick monotonic sequences $a_n,b_n \in ]a,b[$  such that $\lim\limits_{n \rightarrow \infty} a_n=a$ and $\lim\limits_{n \rightarrow \infty} b_n=b$.
If there is a function $v:]a,b[ \mapsto \mathbb R$ such that 
\begin{equation}
\notag
\lim\limits_{n \rightarrow \infty} v^{[a_n,b_n]} (x) = v(x) 
\end{equation}
for every $x \in ]a,b[$ and every sequences $a_n, \ b_n$ as above, then we denote that function by $v^{[a,b]}$.
Notice that in the case $a=m$ (resp., $b=M$), the definition above does not imply that $\lim\limits_{x \rightarrow a} v^{[a,b]}(x) =0$ (resp., $\lim\limits_{x \rightarrow b} v^{[a,b]}(x) =0$).
We will be specially interested in intervals such that \begin{equation}
\label{Eq local positive curve}
a<b
\qquad \text{and} \qquad
v^{[a,b]} (x) > 0 \ \ \forall x \in ]a,b[ .
\end{equation}
Thus, we introduce the following definition.
\begin{defi}
\label{D maximal interval}
We say that an interval $]a,b[$ with $m<a<b<M$, is \emph{maximal for condition \eqref{Eq local positive curve}} if it satisfies \eqref{Eq local positive curve} and is not a proper subset of any other such interval. 
\\If $a=m$ or $b=M$ (or both), we say that $]a,b[$ is maximal for condition \eqref{Eq local positive curve} if there is a monotonically increasing sequence $] a_n , b_n[ $ with $m < a_n < b_n <M$, such that every $]a_n,b_n[$ satisfies \eqref{Eq local positive curve}, $]a,b[= \bigcup \limits_{n \in \mathbb N} ]a_n,b_n[$, and $]a,b[$ is not a proper subset of any other such interval.
\end{defi}

In the following, $\mathcal L^+$ denotes the set of all Lebesgue points of the function $x \mapsto \frac{\Pi(x)}{\sigma(x)^2}$ such that $\Pi(x)>0$.
$\mathcal L^-$ denotes the set of all Lebesgue points of the function $x \mapsto \frac{\Pi(x)}{\sigma(x)^2}$ such that $\Pi(x)<0$.

Along the paper we will suppose that Assumptions \ref{Assumption Engelbert}, \ref{Assumption r} and \ref{Assumption Pi} hold. We will not mention them again, except in Proposition \ref{P assumption Pi vs Pia}, dealing with equivalence between the Assumptions \ref{Assumption Pi} and \ref{Assumption Pia}, and in Subsection \ref{SS T value function}, where intermediate results are proved under a stronger version of these assumptions.

\section{Main results}\label{S results}
In this section we state our main results without proofs.
Full proofs are postponed to Section \ref{S proofs}.

Our characterization of the value function (Theorem \ref{T value function}) relies on maximal intervals for \eqref{Eq local positive curve} and the corresponding functions $v^{[a,b]}$.
Before stating the main result of the section, we give the following properties of maximal intervals.

\begin{prop}\label{P maximal interval}
The following statements hold true:
\begin{itemize}
\item[a)]Different maximal intervals for \eqref{Eq local positive curve} have empty intersection.
\item[b)]Every $x \in \mathcal L^+$ lies in some maximal interval for condition \eqref{Eq local positive curve}.
Conversely, if $]a,b[$ is maximal for \eqref{Eq local positive curve}, then $]a,b[ \cap \mathcal L^+ \neq \emptyset$.
\item[c)]If $]a,b[$ is maximal for \eqref{Eq local positive curve}, then $v^{[a,b]}$ is well defined and $v^{[a,b]}(x) \geq 0$ for every $x \in I$.
Conversely, if $v$ is a solution of \eqref{Eq ODE} such that $v(x) \geq 0$ for every $x \in I$, and $a<b$ are two consecutive zeroes of $v$, then $]a,b[$ is maximal for \eqref{Eq local positive curve}.
\end{itemize}
\end{prop}

By definition, maximal intervals have positive length.
Since they are pairwise disjoint, this implies that there are at most countably many different maximal intervals for condition \eqref{Eq local positive curve}. Consequently, we have the following characterization of the value function.

\begin{teo}
\label{T value function}
Let $\left\{ ]a_k,b_k[, \ k=1,2, \ldots \right\}$ be the collection of all maximal intervals for condition \eqref{Eq local positive curve}.
\\
The value function \eqref{Eq value function} is
\begin{equation}
\label{Eq solution HJB}
V(x) = \left\{\begin{array}{ll}
v^{[a_k,b_k]}(x) & \text{for } x \in ]a_k,b_k[, \ \ k=1,2, \ldots ,
\\
0 & \text{for } x \in I \setminus \bigcup\limits_k ]a_k,b_k[ .
\end{array} \right.
\end{equation}
\end{teo}

Theorem \ref{T value function} begs for some practical way to identify the maximal intervals for \eqref{Eq local positive curve}.
Proposition \ref{P maximal interval} gives some important information.
We complete it with the following:

\begin{prop}
\label{P V intervals}
For any $a \in I$, $b \in ]a,M]$, $]a,b[$ is maximal for \eqref{Eq local positive curve} if and only if: 
\begin{itemize}
\item[a)]
$v_{a,0}(x) \geq 0 $ for every $x \in I$, and
\item[b)]
there is a sequence $a_n \in ]a, M[ \cap \mathcal L^-$ such that
$\left\{ x> a_n : v_{a_n,0}(x) \leq 0 \right\} \neq \emptyset$ for every $n$ and
\begin{align*}
\lim_{n \rightarrow \infty} a_n = a,
\qquad \text{and} \qquad
\lim_{n \rightarrow \infty} \inf \left\{ x> a_n : v_{a_n,0}(x) \leq 0 \right\} = b .
\end{align*}
\end{itemize}
In that case, $v^{[a,b]} = v_{a,0}$.

For any $b \in I$, $a \in [m,b[$, $]a,b[$ is maximal for \eqref{Eq local positive curve} if and only if: 
\begin{itemize}
\item[c)]
$v_{b,0}(x) \geq 0 $ for every $x \in I$, and
\item[d)]
there is a sequence $b_n \in]m,b[ \cap \mathcal L^-$ such that
$\left\{ x< b_n : v_{b_n,0}(x) \leq 0 \right\} \neq \emptyset$ for every $n$ and
\begin{align*}
\lim_{n \rightarrow \infty} b_n = b,
\qquad \text{and} \qquad
\lim_{n \rightarrow \infty} \sup \left\{ x < b_n : v_{b_n,0}(x) \leq 0 \right\} = a .
\end{align*}
\end{itemize}
In that case, $v^{[a,b]} = v_{b,0}$.
\end{prop}

Fix a interval $]a,b[$ with $m<a<b<M$, maximal for \eqref{Eq local positive curve}.
Due to the Propositions above, we have $v^{[a,b]} = v_{a,0} = v_{b,0}$.
By \eqref{Eq solution system ODE},
$ v^\prime_{a,0} (x) = - \int_a^x \frac{2\Pi (z)}{\sigma (z)^2} \phi_{22} (z,x) dz $.
Hence, the points $a,b$ solve the following set of nonlinear equations
\begin{equation}
\label{Eq a b}
\int_a^b \frac{\Pi (z)}{\sigma (z)^2} \phi_{12} (z,b) dz = 0 ,
\qquad
\int_a^b \frac{\Pi (z)}{\sigma (z)^2} \phi_{22} (z,b) dz = 0 ,
\qquad
a < b.
\end{equation}
If $]a,M[$ is maximal for \eqref{Eq local positive curve} and $a \in I$, then for any sequence $\{b_n \in I\}_{n \in \mathbb N}$ converging to $M$, $a$ solves the equation:
\begin{equation}
\label{Eq a unilateral}
\lim_{n\to \infty}\int_{a}^{b_n}\frac{\Pi (z)}{\sigma (z)^2} \phi_{12} (z,b_n)dz=0.
\end{equation}
Similarly, if $]m,b[$ is maximal for \eqref{Eq local positive curve} and $b\in I$, then for any sequence $\{a_n \in I \}_{n\in\mathbb{N}}$ converging to $m$, $b$ solves the equation:
\begin{equation}
\label{Eq b unilateral}
\lim_{n\to \infty}\int_{a_n}^{b}\frac{\Pi (z)}{\sigma (z)^2} \phi_{12} (z,a_n)=0.
\end{equation}
In Section \ref{S example} we will see that equations \eqref{Eq a b}, \eqref{Eq a unilateral}, \eqref{Eq b unilateral} simplify considerably when $X$ is a geometric Brownian motion.

Theoretically, the value function can be found through the following steps:
\begin{itemize}
\item[(I)]
Find the solutions of \eqref{Eq a b}.
Discard any solutions $(a,b)$ such that $\int_a^x \frac{\Pi (z)}{\sigma (z)^2} \phi_{12} (z,x) dz > 0 $ for some $x \in I$.
\\
This yields at most countably many solutions $(a_k,b_k)$, $k=1,2, \ldots $, and the collection of all the intervals between consecutive zeroes of some $v_{a_k,0}$ is the collection of all maximal intervals for \eqref{Eq local positive curve}, with $a>m$ and $b<M$.
\item[(II)]
If there is some $a \in I$ such that $v_{a,0}(x) \geq 0$ for every $x \in I$, then find 
\begin{align*}
&
\hat a = \inf \left\{ a \in I: v_{a,0}(x) \geq 0 \ \text{for every } x \in I \right\} ,
\qquad
\hat b = \sup \left\{ b \in I: v_{b,0}(x) \geq 0 \ \text{for every } x \in I \right\} .
\end{align*}
If $\hat a >m$, then $]m,\hat a[$ is maximal for \eqref{Eq local positive curve}. 
If $\hat b < M$, then $]\hat b, M[$ is maximal for \eqref{Eq local positive curve}. 
\\
This yields all maximal intervals of type $]m,a[$ or $]b,M[$, if such intervals exist.
\item[(III)]
If for every $a \in I $ there is some $x \in I$ such that $v_{a,0}(x) <0$, then $I$ is maximal for \eqref{Eq local positive curve}.
\end{itemize}

\section{Examples}\label{S example}
R{\"u}schendorf and Urusov \cite{ruschendorf2008class}, and Belomestny, R{\"u}schendorf and Urusov \cite{belomestny2010optimal} characterize the value function \eqref{Eq value function} as the solution of a free boundary problem, assuming that the function $\Pi$ is of ``two sided form'' and the support of the value function is an interval $[a,b]$, with $m<a<b<M$.
The results in the previous section do not require any particular structure neither for $\Pi$ nor for the value function.

In Example 1 we discuss a case where $\Pi$ is of ``two sided form'' but the value function may fail to satisfy the assumption in \cite{ruschendorf2008class,belomestny2010optimal}, depending on parameters.
Example 2 deals with a simple case where $\Pi$ is not of ``two sided form''.
In both Examples, we assume that the process $X$ is a geometric Brownian motion and the discount rate is constant. 
This means that $\alpha(x)=\alpha x$, $\sigma(x)=\sigma x$, $r(x)=r$, with $\alpha,\sigma, r$ constants and $I=]0,+\infty[$. 
Moreover, $P_x\{\tau_I=+\infty\}=1$, for every $x \in ]0,+\infty[$, where $P_x$ denotes the conditional probability in $X_0=x$.
The matrix $A(x)$ is
\begin{equation}
\notag
 A(x)=\begin{pmatrix}
   0& 1 \\
  \frac{2r}{\sigma^2x^2}& - \frac{2\alpha}{\sigma^2x}
 \end{pmatrix} .
\end{equation}
Before presenting the examples, we will discuss the fundamental solution $\Phi$ associated with this matrix.

The ordinary differential equation \eqref{Eq ODE} takes the form
\begin{equation}\label{Eq ODE GBM}
rv(x)-\alpha xv'(x)-\frac{\sigma^2}{2}x^2v''(x)-\Pi(x)=0.
\end{equation}
Using the change of variable $x=e^z$ and $y(z)=v(e^z)$, this reduces to the equation with constant coefficients:
\begin{equation}\label{Eq reduced equation}
ry(z)-\left(\alpha-\frac{\sigma^2}{2}\right)y'(z)-\frac{\sigma^2}{2}y''(z)-\Pi(e^z)=0.
\end{equation}
The fundamental matrix $\Phi$ is characterized by the roots of the characteristic polynomial of \eqref{Eq reduced equation}
\begin{equation}
\notag
P(d)=-\frac{\sigma^2}{2}d^2-\left(\alpha-\frac{\sigma^2}{2}\right)d+r.
\end{equation}
Let $d_1$ and $d_2$ be the roots of $P$. The model's data, $(r,\alpha,\sigma)$ may be parametrized by $(d_1,d_2,\sigma)$ through the relations
$$\alpha=\frac{\sigma^2}{2}(1-d_1-d_2) \quad\text{and}\quad r=-\frac{\sigma^2}{2}d_1d_2.
$$
Three different cases must be considered: (i) $d_1=\overline{d_2}\in\mathbb{C}\setminus\mathbb{R}$, (ii) $d_1=d_2\in\mathbb{R}$ and (iii) $d_1,d_2\in\mathbb{R}$ with $d_1 \neq d_2$.

{\bf Case (i):} Let $d_1=a+ib$, $d_2=a-ib$.
The fundamental matrix associated to the equation \eqref{Eq ODE GBM} is 
\begin{equation}
\notag
\Phi (x,y)=\left(\frac{y}{x}\right)^a\left( \begin{array}{cc}
\frac{b\cos\left(b \log\left(\frac{y}{x}\right)\right)-a\sin\left(b\log\left(\frac{y}{x}\right)\right)}{b}&\frac{x\sin\left(b\log\left(\frac{y}{x}\right)\right)}{b}\\
-\frac{\left(a^2+b^2\right)\sin\left(b\log\left(\frac{y}{x}\right)\right)}{by} &\frac{x}{y}\frac{b\cos\left(b \log\left(\frac{y}{x}\right)\right)+a\sin\left(b\log\left(\frac{y}{x}\right)\right)}{b}
\end{array} \right).
\end{equation}
Thus, the function $y\to\phi_{1,2}(x,y)$ has infinitely many zeroes. 
Therefore, in light of Proposition \ref{P sign phi_12}
\begin{equation}
\label{Eq Z01}
\mathbb{E}_x\left[\int_0^{+\infty} e^{-rt}\Pi^+(X_t)dt\right]=+\infty
\end{equation}
for every $x \in ]0,+\infty[$ and every measurable $\Pi$ such that the set $\left\{ x>0: \Pi(x)>0 \right\}$ has strictly positive Lebesgue measure. 
Thus, Assumption \ref{Assumption Pi}/\ref{Assumption Pia} fails and the problem is either trivial or ill-posed, as explained in Section \ref{S problem setting}.

{\bf Case (ii):} Let $d=d_1=d_2$.
In this case, the fundamental matrix is
\begin{equation}
\notag
  \Phi (x,y)=\left(\frac{y}{x}\right)^d\left( \begin{array}{cc}
\left(1-d\log\left(\frac{y}{x}\right)\right)& x\log\left(\frac{y}{x}\right)\\
-\frac{d^2}{y} \log\left(\frac{y}{x}\right) & \frac{x}{y}\left(1+d\log\left(\frac{y}{x}\right)\right)
\end{array} \right).
\end{equation}
For every $x \in ]0,+\infty[$, the function $y \mapsto \phi_{12}(x,y)$ has one unique zero.
However, a tedious but trivial computation shows that
$$
\lim_{n\to +\infty}v^{[\frac{1}{n},n]}(x)=+\infty\quad \text{for every } x>0
$$
whenever $\Pi$ is non-negative and the set $\{x>0:\Pi(x)>0\}$ has strictly positive Lebesgue measure.
Thus, \eqref{Eq Z01} holds also in this case and therefore the problem is again either trivial or ill-posed.

{\bf Case (iii):} Without loss of generality, we assume that $d_1<d_2$. 
The fundamental matrix is 
\begin{equation}\label{Eq phi d1<d2}
\Phi (x,y)=\left( \begin{array}{cc}
\frac{d_2\left(\frac{y}{x}\right)^{d_1}-d_1\left(\frac{y}{x}\right)^{d_2}}{d_2-d_1}
 & x\frac{\left(\frac{y}{x}\right)^{d_2}-\left(\frac{y}{x}\right)^{d_1}}{d_2-d_1}\\
d_1d_2\frac{\left(\frac{y}{x}\right)^{d_1-1}-\left(\frac{y}{x}\right)^{d_2-1}}{(d_2-d_1)x} & \frac{d_2\left(\frac{y}{x}\right)^{d_2-1}-d_1\left(\frac{y}{x}\right)^{d_1-1}}{d_2-d_1}
\end{array} \right).
\end{equation}
Like in case (ii), for every $x >0$ the function $y \mapsto \phi_{12}(x,y)$ has one unique zero.
Thus, the discussion above leaves this as the only interesting case.
For this reason, in both examples below we will assume that $d_1,d_2 \in \mathbb R$, with $d_1 < d_2$.

Notice that in case (iii), substitution of \eqref{Eq phi d1<d2} in \eqref{Eq v_ad}, yields
\begin{equation}
\label{Eq v_a0 GBM}
v_{a,0}(x) =
\frac{-2}{\sigma^2(d_2-d_1)}\int_a^x\frac{\left(\frac{x}{z}\right)^{d_2}-\left(\frac{x}{z}\right)^{d_1}}{z}\Pi(z)dz .
\end{equation}
The equations \eqref{Eq a b} reduce to
\begin{equation}
\label{Eq a b gbm}
\int_a^b z^{-d_2-1}\Pi (z) dz = 0 ,
\qquad
\int_a^b z^{-d_1-1}\Pi (z) dz = 0 ,
\qquad
a < b,
\end{equation}
and the equations \eqref{Eq a unilateral}, \eqref{Eq b unilateral} become 
\begin{equation}
\label{Eq a b gbm unilateral}
\int_a^{+\infty} z^{-d_2-1}\Pi (z) dz = 0 ,
\qquad
\int_0^b z^{-d_1-1}\Pi (z) dz = 0 ,
\end{equation}
respectively.
Notice that \eqref{Eq v_a0 GBM}--\eqref{Eq a b gbm}--\eqref{Eq a b gbm unilateral} show that the inverse volatility $\frac{1}{\sigma^2}$ acts as multiplicative parameter in the value function.

\underline{Example 1:} 
\label{Example 1}
Fix $0< x_1 < x_2 < +\infty$, and let $\Pi$ be the piecewise constant function
\begin{equation}
\notag
\Pi(x)=2\chi_{[x_1,x_2]}(x)-1, \quad \text{for all } x>0.
\end{equation}
This function is of ``two sided form'' in the sense of Belomestny, R{\"u}schendorf and Urusov \cite{belomestny2010optimal}. 

Due to Proposition \ref{P maximal interval}, $]0,+\infty[$ contains one unique maximal interval for \eqref{Eq local positive curve}, and it contains the interval $]x_1,x_2[$.
Due to \eqref{Eq v_a0 GBM}, for any $a \in ]0,x_1[$, 
\begin{equation}
\notag 
v_{a,0}(x)=
\frac{2}{\sigma^2(d_2-d_1)}\left(x^{d_1}G_1(x)-x^{d_2}G_2(x)\right)
\end{equation}
with
$$
G_i(x)=\begin{cases}
\frac{-1}{d_i}\left(a^{-d_i}-x^{-d_i}\right)& \text{for } x<x_1\\
\frac{-1}{d_i}\left(a^{-d_i}-2x_1^{-d_i}+x^{-d_i}\right)& \text{for } x_1\leq x\leq x_2\\
\frac{-1}{d_i}\left(a^{-d_i}-2x_1^{-d_i}+2x_2^{-d_i}-x^{-d_i}\right)& \text{for } x> x_2,
\hspace{20mm} i=1,2.
\end{cases}
$$
From this, it can be checked that if $d_2>0$, then for every sufficiently small $a>0$ we have $v_{a,0}(x) >0$ for every $x \neq a$. Therefore, the interval $]0,x_1]$ is not contained in the maximal interval for \eqref{Eq local positive curve}.
A similar argument applied to the function $v_{b,0}$ with $b>x_2$ shows that if $d_1<0$ then the interval $[x_2,+\infty[$ is not contained in the maximal interval for \eqref{Eq local positive curve}.
Therefore, for any $d_1<d_2$, $]0,+\infty[$ cannot be the maximal interval.
If $d_1 < 0 < d_2$ then the maximal interval $]a,b[$ must be such that $0<a<x_1<x_2<b < +\infty$.

To see that in the case $d_1<d_2 <0$ the maximal interval can be either $]a,b[$ with $0 < a <x_1 <x_2 < b <+\infty$ or $]0,b[$ with $x_2 < b < +\infty$, we consider the case $d_1=-2$, $d_2= -1$, where explicit computations are trivial.
Notice that for $\Pi$ of ``two sided form''  and for $0<a<b<+\infty$, $]a,b[$ is maximal if and only if $(a,b)$ solves \eqref{Eq a b gbm}.
For $d_1=-2$, $d_2=-1$, it is easy to check that \eqref{Eq a b gbm} admits a solution with $0<a<b<+\infty$ if and only if $x_1 > \frac{x_2}3$, and in that case
\[
a = \frac{3 x_1 - x_2} 2 ,
\qquad
b = \frac{3x_2-x_1}2 .
\]
If $x_1 < \frac {x_2}3$, the maximal interval is $]0,b[$, with $b$ satisfying the second equality in \eqref{Eq a b gbm unilateral}, that is 
\[
b= \sqrt{2 \left( x_2^2- x_1^2 \right)} .
\]
Therefore, the value function is
\begin{align*}
V(x) = &
\left\{ 
\begin{array}{ll}
v_{a ,0}(x) = v_{b,0}(x), & \text{for } x \in \left[ a , b \right]
\\
0, &  \text{for } x \notin \left[ a , b \right]
\end{array}
\right. 
& 
\text{if } x_1 > \frac{x_2}{3} ,
\\
V(x) = &
\left\{ 
\begin{array}{ll}
v_{ b ,0}(x) , & \text{for } x \leq b
\\
0, &  \text{for } x > b
\end{array}
\right. 
&
\text{if } x_1 < \frac{x_2}{3} ,
\end{align*}
with $a,b$ given by the expressions above.
In the second case, the value function is not supported in a compact subinterval of $]0,+\infty[$. 
Thus, this is an example of a problem that is not solved by the results in \cite{ruschendorf2008class,belomestny2010optimal}. 
Graphs of the value function for both cases are shown in Figure \ref{Figure 1}.
Notice that the case $d_1<d_2<0$ corresponds to a negative discount rate and the value function is unbounded.

\begin{figure}[h!]
\includegraphics[width=75mm]{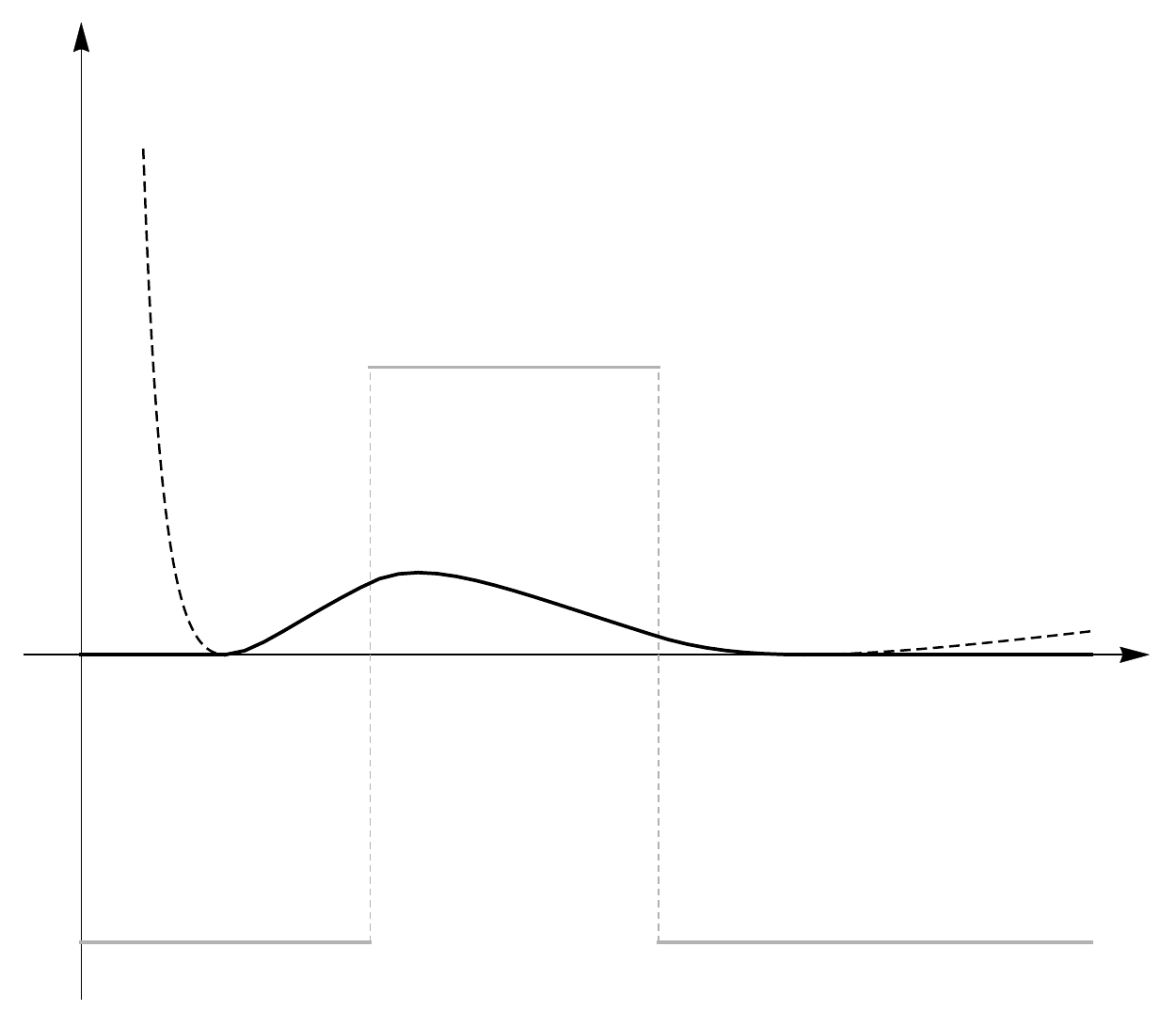} \hspace{5mm}
\includegraphics[width=75mm]{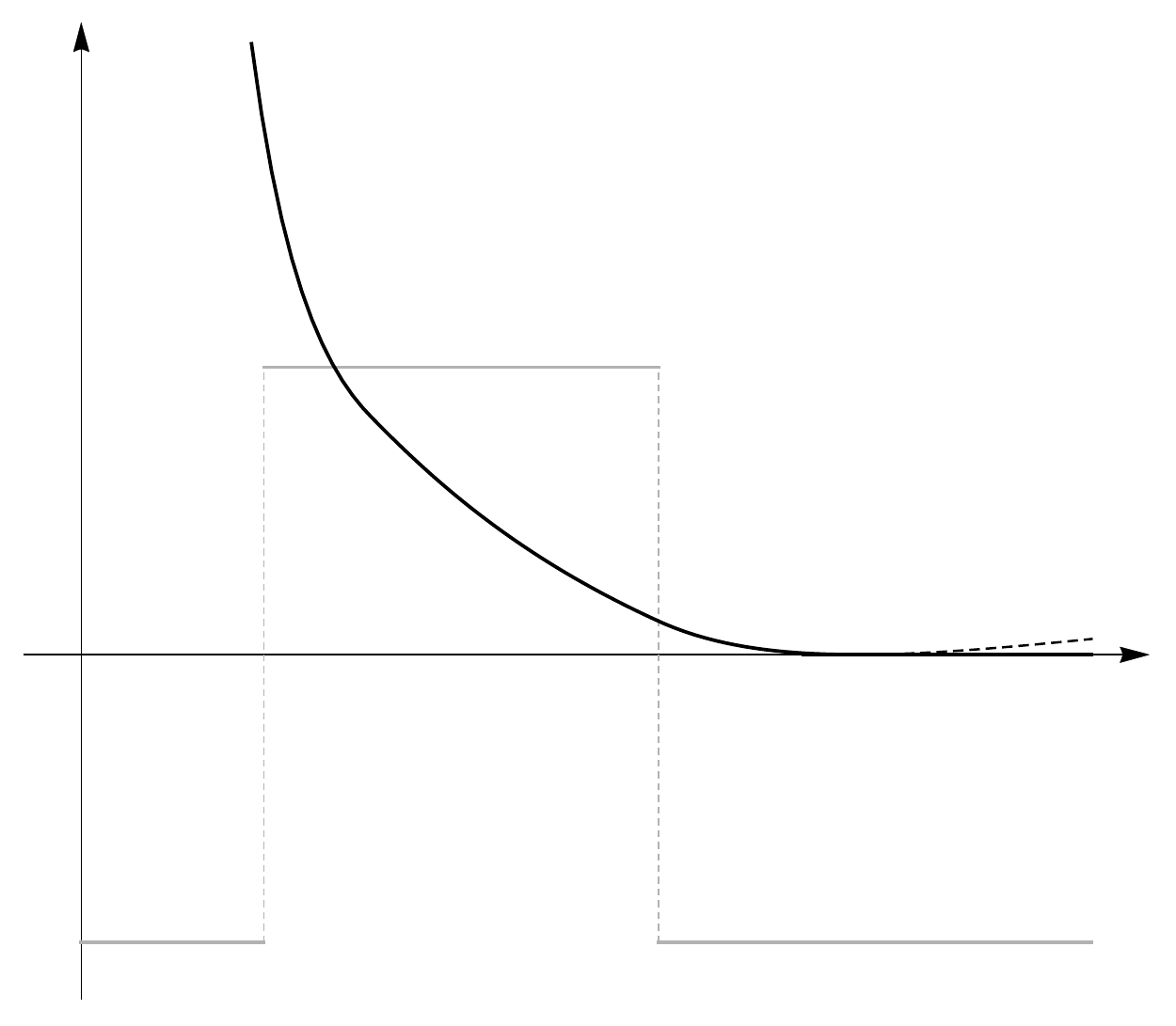}
%
\caption{\label{Figure 1}Value functions for Example 1. The grey lines represent the functions $\Pi$. The black lines represent the value functions $V$. The  dashed lines represent the functions $v_{b,0}$. 
Left-hand picture: $x_1=1$, $x_2=2$. Right-hand picture: $x_1=\frac{19}{30}$, $x_2=2$. In both cases, $\sigma = 1$, $d_1=-2$, $d_2=-1$. Figures drawn to the same scale.
}
\end{figure}

Similar examples with $0<d_1<d_2$ showing that the maximal interval can be either $]a,b[$, with $0<a<x_1<x_2 < b < +\infty$, or $]a,+\infty[$, with $0< a <x_1$ can easily be constructed.

\underline{Example 2:} \label{Example 2}
Fix $0<x_1<x_2<x_3<x_4<+\infty$, and let $\Pi$ be the piecewise constant function
\[
\Pi(x) = 2 \chi_{[x_1,x_2]}(x) + 2 \chi_{[x_3,x_4]}(x) -1 .
\]
Thus, $\Pi$ is positive in two separate intervals.
This is the case discussed in the remarks following Theorem 2.3 of R{\"u}schendorf and Urusov \cite{ruschendorf2008class}, and Theorem 2.2 of Belomestny, R{\"u}schendorf and Urusov \cite{belomestny2010optimal}. 
To discuss this case, we introduce the functions
\[
\Pi_1(x) = 2 \chi_{[x_1,x_2]}(x) -1 ,
\qquad
\Pi_2(x) = 2 \chi_{[x_3,x_4]}(x) -1 .
\]
Let $V$, $V_1$, $V_2$ be the value functions corresponding to $\Pi$, $\Pi_1$, $\Pi_2$, respectively, and let $v_{a,0}$,  $v_{a,0}^1$, $v_{a,0}^2$ be the corresponding functions defined by \eqref{Eq v_a0 GBM}.

In \cite{ruschendorf2008class,belomestny2010optimal} it is remarked that if the support of $V_1$ is an interval $[a,b]$ with $0 <a < b < +\infty$, then $V_1$ solves both the free-boundary problem corresponding to $\Pi_1$ and the free-boundary problem corresponding to $\Pi$, but $V_1$ may coincide or not with $V$ in $[a,b]$.
We will show that the results in Section \ref{S results} above easily distinguish these cases.

Suppose that $d_2>0$ (the case $d_1<0$ is analogous).
From Example 1, there are constants $0<a_1 <x_1 <  x_2 < b_1 \leq +\infty$ such that:
\begin{align*}
V_1(x) = &
\left\{ 
\begin{array}{ll}
v^1_{a_1,0}(x), & \text{for } x \in ]a_1,b_1[ ,
\\
0, & \text{for } x \notin ]a_1,b_1[ .
\end{array}
\right. 
\end{align*}
Since $]a_1,b_1[$ is maximal for \eqref{Eq local positive curve} with respect to $\Pi_1$, $v^1_{a_1,0}$ is non-negative in $]0,+\infty[$.
It is easy to check that $v_{a_1,0}$ coincides with $v_{a_1,0}^1$ in the interval $]0,x_3]$ but these functions are distinct in the interval $]x_3,+\infty[$.
Thus, it may happen that $v_{a_1,0}(x) < 0 $ for some $x>x_3$.
In that case, the Proposition \ref{P maximal interval} shows that $]a_1,b_1[$ is not maximal with respect to $\Pi$ and therefore $v_{a_1,0}$ does not coincide with the value function $V$ in $[a_1,b_1]$.
The Figure \ref{Figure 2} shows an example of this configuration.
Conversely, if $v_{a_1,0}(x) \geq 0 $ for every $x \in ]0,+\infty[$, then $]a_1,b_1[$ is maximal with respect to $\Pi$ and $V$ coincides with $v^1_{a_1,0}$ in the interval $[a_1,b_1]$.
The right-hand picture in Figure \ref{Figure 3} shows an example of this configuration.
\begin{figure}[h!]
\includegraphics[width=120mm]{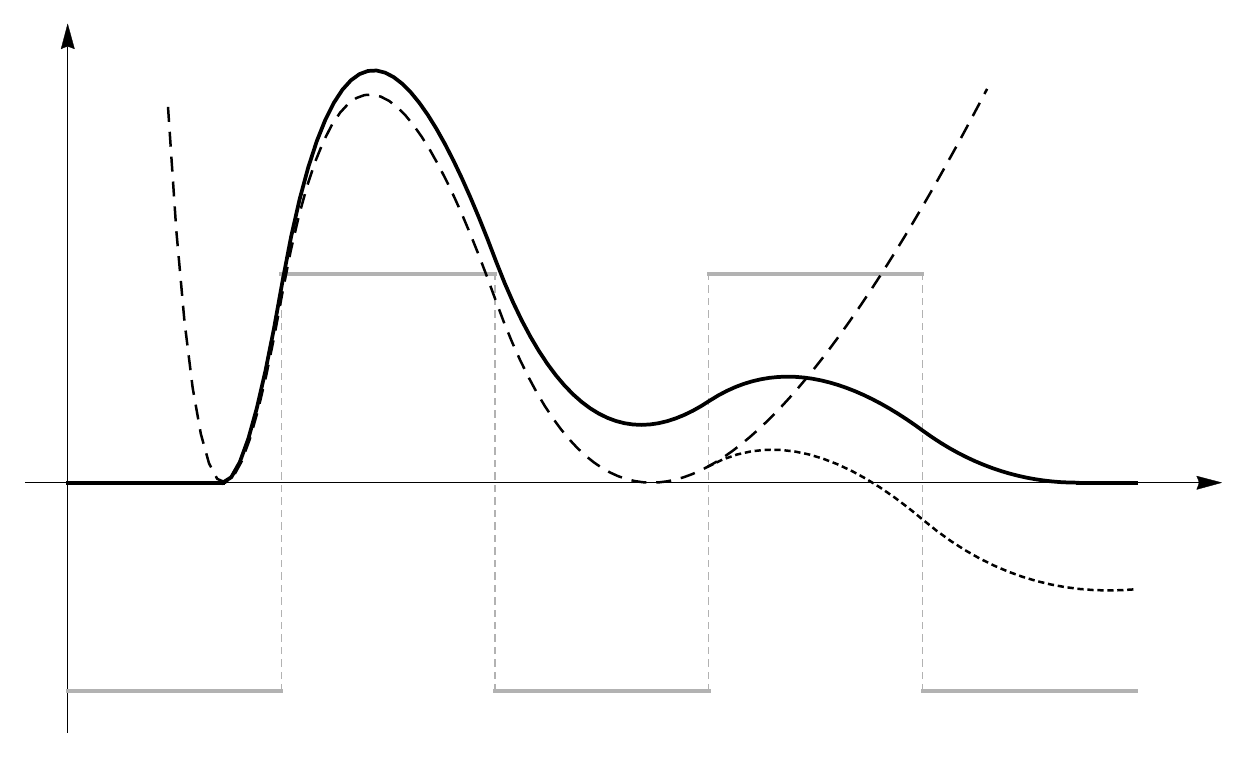} 
\caption{\label{Figure 2}Example 2: Value functions for $\Pi_1$ and $\Pi$. 
The grey line represents the function $\Pi$.
The dashed line represents the function $v^1_{a_1,0}$.
The dotted line represents the function $v_{a_1,0}$.
The black line represent the value function for $\Pi$.
Parameters: $x_1=1$, $x_2=2$, $x_3=3$, $x_4= 4$, $\sigma = \frac 1 3$, $d_1=-1$, $d_2=1$. 
}
\end{figure}

Another way to see the same phenomenon is as follows.
Let $]a_1,b_1[$, $]a_2,b_2[$ be the maximal intervals with respect to $\Pi_1$ and $\Pi_2$, respectively (by Example 1, these intervals exist and are unique, with $a_1,a_2 >0$).
If $a_2 < b_1$, then the Proposition \ref{P maximal interval} states that these intervals cannot be maximal with respect to $\Pi$. Hence, the maximal interval for $\Pi$ must be a larger interval $]a,b[$ containing $]a_1,b_1[ \cup ]a_2,b_2[$.
Conversely, if $a_2 \geq b_1$, then $]a_1,b_1[$, $]a_2,b_2[$ are both maximal with respect to $\Pi$, and therefore the value function is
\[
V(x) = 
\left\{ 
\begin{array}{ll}
v^1_{a_1,0}(x) = v_{a_1,0}(x) , & \text{for } x \in [a_1,b_1],
\\
v^2_{a_2,0}(x) = v_{a_2,0}(x) , & \text{for } x \in [a_2,b_2],
\\
0, & \text{for } x \notin [a_1,b_1] \cup [a_2,b_2] .
\end{array}
\right.
\]
The Figure \ref{Figure 3} shows an example with $a_2 < b_1$ and an example with $a_2 > b_1$.
\begin{figure}[h!]
\includegraphics[width=75mm]{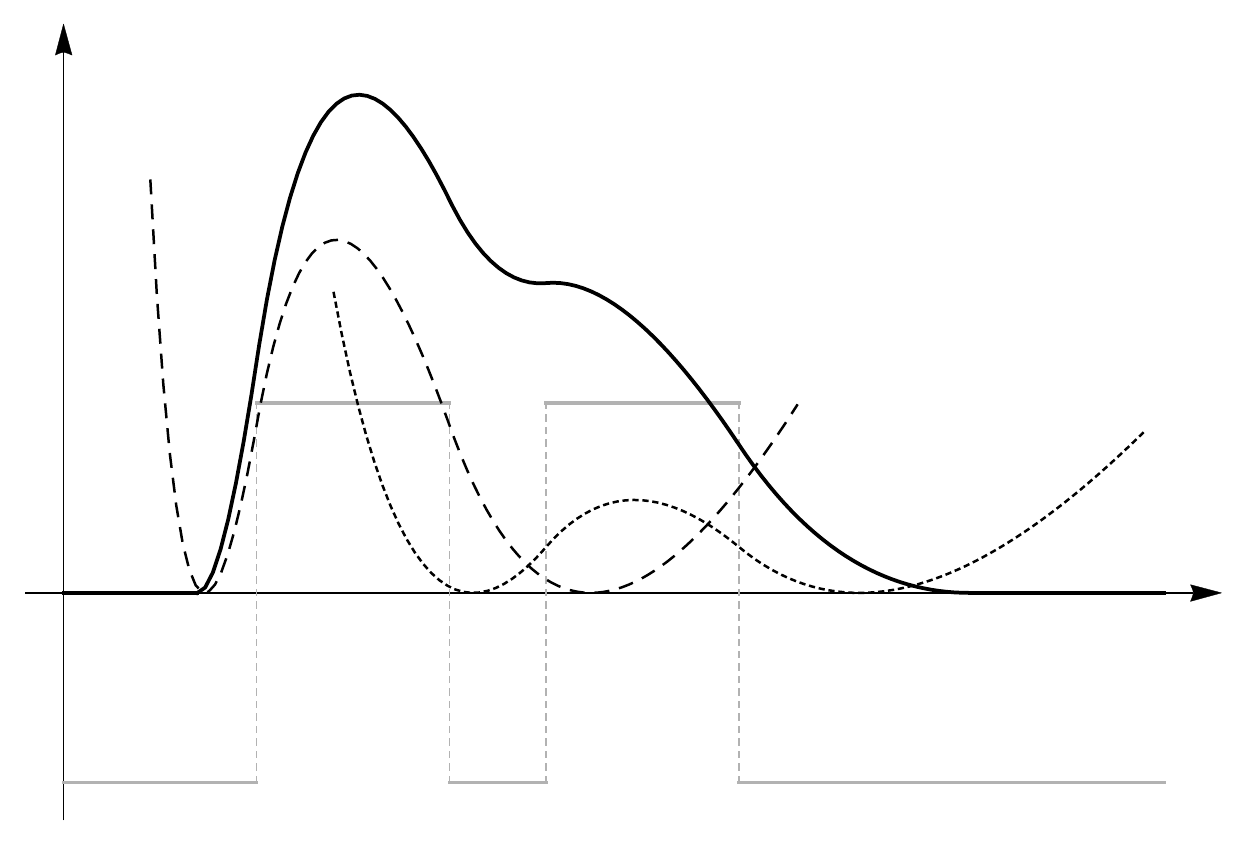} 
\hspace{5mm}
\includegraphics[width=75mm]{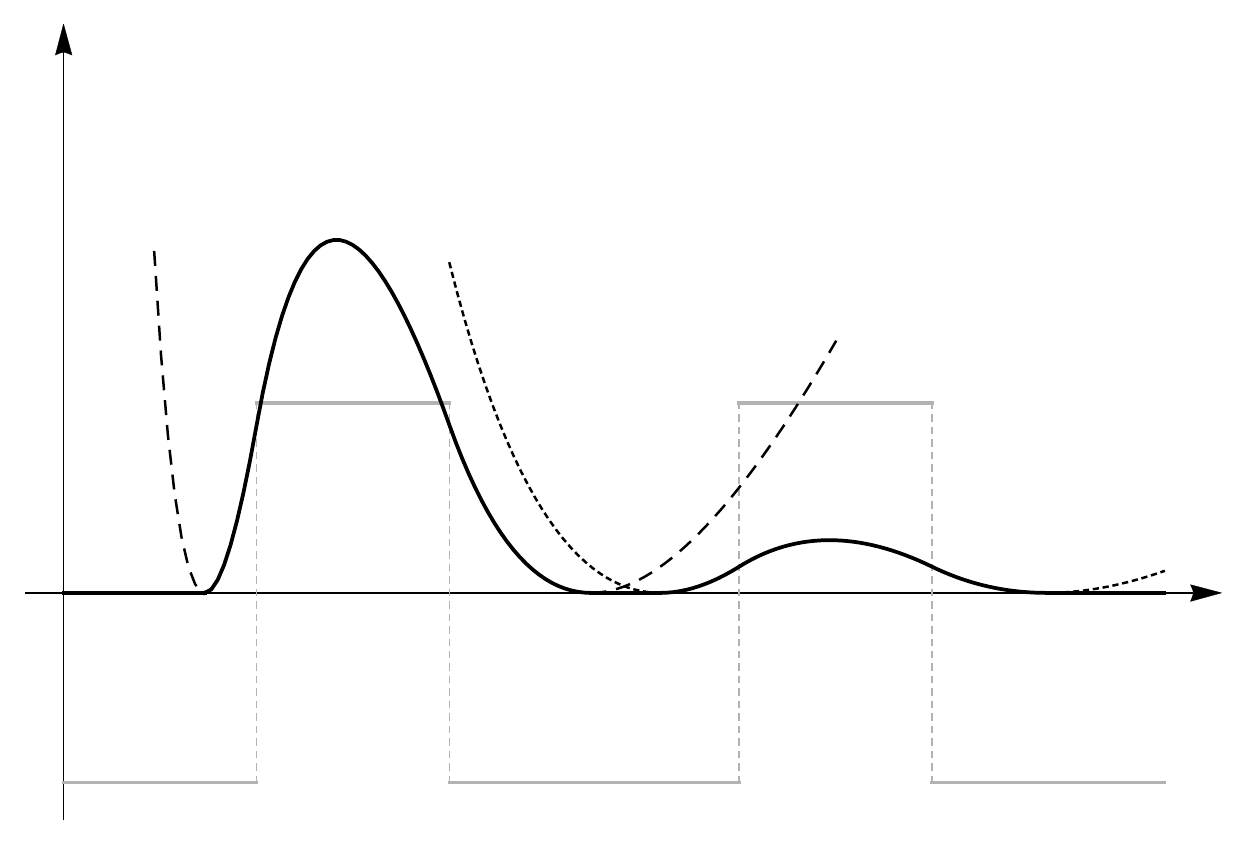}
\caption{\label{Figure 3}Example 2: 
Value functions for $\Pi_1$, $\Pi_2$, and $\Pi$. 
Grey lines represent the function $\Pi$.
Dashed lines represent the functions $v^1_{a_1,0}$.
Dotted lines represent the functions $v^2_{a_2,0}$.
Black lines represent the value function for $\Pi$.
Left-hand picture: $x_1=1$, $x_2=2$, $x_3=2.5$, $x_4= 3.5$, $b_1 \approx 2.73 > a_2 \approx 2.12$.
Right-hand picture: $x_1=1$, $x_2=2$, $x_3=3.5$, $x_4= 4.5$, $b_1 \approx 2.73 < a_2 \approx 3.09$.
In both cases, $\sigma = \frac 1 3$, $d_1=-1$, $d_2=1$. 
Figures drawn to the same scale.
}
\end{figure}
Notice that the value function depicted on the right-hand side of Figure \ref{Figure 3} is not $F$-concave with respect to any strictly increasing function $F$.
This shows that the result by Dayanik and Karatzas \cite{dayanik2003optimal} for the problem \eqref{Eq terminal criterion}--\eqref{Eq discount factor}--\eqref{Eq SDE} does not extend to the problem \eqref{Eq outcome}--\eqref{Eq discount factor}--\eqref{Eq SDE}.

\section{Proofs}\label{S proofs}

\subsection{Some preliminary results}\label{SS preliminary results}

The results in Section \ref{S results} depend critically on the following Proposition.

\begin{prop}
\label{P sign phi_12}
$\phi_{12}(a,b)>0 $ for every $a,b \in I$, with $a< b$.
\end{prop}

The proof of this Proposition requires several intermediate lemmata, which we formulate and prove below.
As a corollary, we will prove the following.

\begin{prop}
\label{P assumption Pi vs Pia}
Under Assumptions \ref{Assumption Engelbert} and \ref{Assumption r}, Assumptions \ref{Assumption Pi} and \ref{Assumption Pia} are equivalent.
\end{prop}

Another easy corollary of Proposition \ref{P sign phi_12} is the following Lemma, that will be useful to several arguments in the next subsections.

\begin{lemma}
\label{L no double cross}
If $u,v$ are solutions of \eqref{Eq ODE}, and there are two points $a,b \in I$ such that
\begin{equation}
\notag
u(a) = v(a), \qquad
u(b) = v(b), \qquad
a \neq b ,
\end{equation}
then $u \equiv v$.
\end{lemma}

\begin{proof}
Follows immediately from Proposition \ref{P sign phi_12} and equality \eqref{Eq v nonhomogeneous}.
\end{proof}

To prove Proposition \ref{P sign phi_12}, we start with Lemmata \ref{L sign phi_12} and \ref{L roots phi_12}, which contain some simple properties of the fundamental solution $\Phi$.

\begin{lemma}
\label{L sign phi_12}
For every $a \in I$, the following statements are true:
\begin{itemize}
\item[a)]
There is some $b \in ]a,M[$ such that $\phi_{12}(a,x)>0$ for every $x \in ]a,b[$.
\item[b)]
If there is some $x \in ]a,M[$ such that $\phi_{12}(a,x)=0$, then $\phi_{11}(a,b)<0$ and $\phi_{22}(a,b)<0$ for $b= \min \left\{ x>a: \phi_{12}(a,x) =0 \right\}$.
\item[c)]
If the function $x \mapsto \phi_{12}(a,x)$ is strictly positive in the interval $]a,b[$, then the function $x \mapsto \phi_{12}(x,b)$ is strictly positive in the interval $]a,b[$.
\end{itemize}
\end{lemma}

\begin{proof}
The statement {\it (a)} follows immediately from the fact that 
\begin{equation}
\notag
\frac{\partial }{\partial x}\phi_{12}(a,x) = \phi_{22}(a,x) \qquad \forall x \in I,
\end{equation}
and $\phi_{12}(a,a)=0$, $\phi_{22}(a,a)=1$.

To prove the statement {\it (b)}, notice that $\phi_{22}(a,b) = \frac \partial{\partial x} \phi_{1,2}(a,b) \leq 0$.
Since $\mathrm{det} \Phi (a,x) >0$ for every $x \in I$, $\phi_{12}(a,b)=0$ implies $\phi_{11}(a,b) \phi_{22}(a,b) >0$, and the statement follows.

Finally, to prove the statement {\it (c)}, we start by recalling that $\Phi(a,b) = \Phi(x,b) \Phi(a,x)$.
Therefore:
\begin{align}
\label{Z03}
\phi_{12}(x,b) = &
\frac{1}{\mathrm{det}\Phi(a,x)} \left( \phi_{12}(a,b) \phi_{11}(a,x) - \phi_{11}(a,b) \phi_{12} (a,x) \right)  .
\end{align}
If $\phi_{12}(a,b) >0$, this reduces to
\begin{align*}
\phi_{12}(x,b) = &
\frac{\phi_{12}(a,b) \phi_{12} (a,x)}{\mathrm{det}\Phi(a,x)} \left(  \frac{\phi_{11}(a,x)}{\phi_{12}(a,x)} - \frac{\phi_{11}(a,b)}{\phi_{12}(a,b)} \right) .
\end{align*}
A simple computation shows that
\begin{equation}
\notag
\frac{\partial }{\partial x} \frac{\phi_{11}(a,x)}{\phi_{12}(a,x)} = 
- \frac{\mathrm{det} \Phi(a,x)}{\phi_{12}(a,x)^2} <0 .
\end{equation}
Hence, the function $x \mapsto \frac{\phi_{11}(a,x)}{\phi_{12}(a,x)} $ is strictly decreasing in $]a,b]$ and therefore $\phi_{12}(x,b) >0 $ for every $x \in ]a,b[$. If $\phi_{12}(a,b)=0$, then the equality \eqref{Z03} reduces to
\begin{align*}
\phi_{12}(x,b) = & 
- \frac{\phi_{11}(a,b) \phi_{12} (a,x)}{\mathrm{det}\Phi(a,x)} .
\end{align*}
By the statement {\it (b)}, $\phi_{11}(a,b) <0$ and therefore, $\phi_{12}(x,b)>0$ for every $x \in ]a,b[$.
\end{proof}

\begin{lemma}
\label{L roots phi_12}
Suppose that there is some $a,b \in I$ such that $a<b$ and $\phi_{12}(a,b)=0$. 
Then, for every $a^\prime \in ]m,a[$ there is some $b^\prime \in [a , b[$ such that $\phi_{12}(a^\prime,b^\prime)=0$. Similarly, for every $b^\prime \in ]b,M[$ there is some $a^\prime \in ]a , b]$ such that $\phi_{12}(a^\prime,b^\prime)=0$.
\end{lemma}

\begin{proof}
Fix $a,b \in I$ such that $a<b$ and $\phi_{12}(a,b)=0$.
Without loss of generality, we may assume that $\phi_{12}(a,x)>0$ for every $x \in ]a,b[$ (take a subinterval, if necessary).

Fix $a^\prime < a$.
Since $x \geq a$, $\Phi(a^\prime ,x) = \Phi(a,x) \Phi(a^\prime , a ) $, we have 
\begin{equation}
\notag
\phi_{12}(a^\prime , x) = 
\phi_{11}(a,x) \phi_{12}(a^\prime, a) + 
\phi_{12}(a,x) \phi_{22}(a^\prime, a)
\qquad \forall x \in ]a,b[ .
\end{equation}
By the statement {\it (b)} of Lemma \ref{L sign phi_12}, this must be negative for every $x $ sufficiently close to $b$ if $\phi_{12}(a^\prime , a ) > 0$.
Thus, $\phi_{12}(a^\prime,x)$ must have a zero in $[a, b[$.

Now, fix $b^\prime >b$.
Since $\Phi(a,b^\prime) = \Phi(b,b^\prime ) \Phi(a,b)$, $\phi_{ 12}(a,b)=0$ implies
$\phi_{12} (a, b^\prime ) = \phi_{12}(b,b^\prime) \phi_{22}(a,b)$.
By the statement {\it (b)} of Lemma \ref{L sign phi_12}, this must be negative if $\phi_{12}(b,b^\prime ) > 0$.
Hence the function $x \mapsto \phi_{12}(x,b^\prime )$ must have a zero in $]a,b]$.
\end{proof}

The Lemma \ref{L existence of positive solution} relates the sign of $\phi_{12}$ with the sign of solutions of equations of type \eqref{Eq ODE}.
To prove the Proposition \ref{P sign phi_12},
we need to consider such equations with different functions instead of $\Pi$.
That is, we consider variants of equation \eqref{Eq ODE} of the type:
\begin{equation}
\label{Eq ODE g}
r(x) v(x)-\alpha(x) v^\prime (x)-\frac{\sigma(x)^2}{2} v^{\prime \prime} (x)- g(x)=0,
\end{equation}
where $g: I \mapsto \mathbb R$ is a measurable function such that $\frac{g}{\sigma^2} $ is locally integrable in $I$ with respect to the Lebesgue measure.

\begin{lemma}
\label{L existence of positive solution}
Let $g:I \mapsto [0,+\infty[$ be a measurable function such that $\frac{g}{\sigma^2}$ is locally integrable, and $\int_a^b \frac{g(z)}{\sigma(z)^2} dz >0$.
Equation \eqref{Eq ODE g} admits a non-negative solution in the interval $[a,b] \subset I $ if and only if $\phi_{12}(a,x) >0$ for every $x \in ]a,b]$.
\end{lemma}

\begin{proof}
The function
\begin{equation}
\notag
v(x) = K \phi_{12}(a,x) - \int_a^{x}\frac{2g(z)}{\sigma(z)^2}\phi_{12}(z,x)dz 
\end{equation}
is a solution of \eqref{Eq ODE g}. 
For sufficiently large $K \in ]0,+\infty[$, it is non-negative in $[a,b]$, provided $\phi_{12}$ is strictly positive in $]a,b]$.

Now, suppose that there is some $x_0 \in ]a,b]$ such that $\phi_{12}(a,x_0) \leq 0$.
Without loss of generality, we may assume that $x_0 = b = \min \{ x >a : \phi_{12}(a,x) = 0 \}$ (take a subinterval on $[a,b]$, if necessary).
Fix $v$, a solution of \eqref{Eq ODE g}.
By \eqref{Eq v nonhomogeneous},
\begin{equation}
\notag
v(b) = v(a) \phi_{11}(a,b) - \int_a^b \frac{2g(z)}{\sigma(z)^2}\phi_{12}(z,b)dz .
\end{equation}
The Lemma \ref{L sign phi_12} states that $\phi_{11}(a,b) <0$ and $\phi_{12}(z,b) >0 $ for every $z \in ]a,b[$.
Therefore, $v(b)<0$.
\end{proof}

For any $a,b \in I$, with $a<b$, we define the stopping times
\begin{align*}
\tau_{[a,b]} = & \inf \left\{ t \geq 0 : X_t \notin ]a,b[ \right\} ,
\\
\tau_a = &
\left\{ 
\begin{array}{ll}
\inf \left\{ t \geq 0 : X_t =a \right\} , & \text{if } \left\{ t \geq 0 : X_t =a \right\} \neq \emptyset ,
\\
\tau_I , & \text{if } \left\{ t \geq 0 : X_t =a \right\} = \emptyset .
\end{array}
\right.
\end{align*}
It is clear that $\tau_{[a,b]}$ and $\tau_a$ are admissible stopping times, as defined in Section \ref{S problem setting}.

The following Lemmata \ref{L integrability 1} and \ref{L integrability 2} relate the solutions of equation \eqref{Eq ODE g} with the value of a functional of type \eqref{Eq outcome}.
The results and the arguments in the proofs are similar to many classical results
(see, e.g. 
Dayanik and Karatzas \cite{dayanik2003optimal}, 
R{\"u}schendorf and Urusov \cite{ruschendorf2008class},
Belomestny, R{\"u}schendorf and Urusov \cite{belomestny2010optimal}, 
Lamberton and Zervos \cite{lamberton2013optimal}, and references therein).
However, since similar arguments are used to prove other results below, we outline the argument in the proof of Lemma \ref{L integrability 1}.

\begin{lemma}
\label{L integrability 1}
Let $g:I \mapsto [0,+\infty[$ be a measurable function such that $\frac{g}{\sigma^2}$ is locally integrable.
Let $v$ be a solution of equation \eqref{Eq ODE g}, non-negative in a compact interval $[a,b] \subset I$.
Then
\begin{equation}
\notag
\mathbb E_x \left[ \int_0^{\tau_{[a , b ]}} e^{-\rho_s} g(X_s) ds \right] \leq v(x)
\qquad \forall x \in ]a,b[.
\end{equation}
\end{lemma}

\begin{proof}
Let $\{\theta_n\}_{n\in\mathbb{N}}$ be a sequence of stopping times such that $\theta_n\to \tau_I$ and the stopped process $\{X_{t\wedge\theta_n}\}_{t\geq 0}$ is a semimartingale.
Using the It\=o-Tanaka formula and the occupation times formula (see for example theorem VI.1.5 and corollary VI.1.6 in Revuz and Yor \cite{revuz2013continuous}), we obtain
\begin{align*}
&
e^{-\rho_{\tau_{[a,b]} \wedge \theta_n}} v(X_{\tau_{[a,b]}\wedge\theta_n}) = 
\\ = &
v(x)+
\int_0^{\tau_{[a,b]}\wedge\theta_n} e^{-\rho_s}\left( -rv + \alpha v^\prime + \frac{\sigma^2}{2} v^{\prime \prime} \right) \circ X_s ds +
\int_0^{\tau_{[a,b]}\wedge\theta_n}\left( \sigma v'\right) \circ X_s dW_s =
\\ = &
v(x)-
\int_0^{\tau_{[a,b]}\wedge\theta_n} e^{-\rho_s}g ( X_s) ds +
\int_0^{\tau_{[a,b]}\wedge\theta_n}\left( \sigma v'\right) \circ X_s dW_s .
\end{align*}
Therefore,
\begin{align*}
0 \leq \mathbb E_x \left[ e^{-\rho_{\tau_{[a,b]}\wedge \theta_n}} v \left( X_{\tau_{[a,b]}\wedge \theta_n} \right) \right] =
v(x) - \mathbb E_x \left[ \int_0^{\tau_{[a,b]}\wedge\theta_n} e^{-\rho_s}g ( X_s ) ds \right] .
\end{align*}
Making $n \to \infty$, the Lemma follows from the Lebesgue monotone convergence theorem.
\end{proof}

\begin{lemma}
\label{L integrability 2}
Fix a compact interval $[a,b] \subset I$ such that $\phi_{1,2}(a,x) >0$ for every $x \in ]a,b]$, and let $g: I \mapsto \mathbb R$ be a measurable function such that $\frac{g}{\sigma^2}$ is locally integrable with respect to the Lebesgue measure. 
\\
If $v$ is the unique solution of equation \eqref{Eq ODE g} with boundary conditions $v(a)=v(b)=0$, then
\begin{equation}
\notag
\mathbb E_x \left[ \int_0^{\tau_{[a , b ]}} e^{-\rho_s} g(X_s) ds \right] =v(x) 
\qquad
\forall x \in ]a,b[.
\end{equation}
\end{lemma}

\begin{proof}
Fix $[a,b]$ as above.
By the argument used in the proof of Lemma \ref{L integrability 1}, there is a sequence of stopping times $\{\theta_n\}_{n\in\mathbb{N}}$ such that $\theta_n\to \tau_I$ and
\begin{align}
\label{Z01}
\mathbb E_x \left[ e^{-\rho_{\tau_{[a,b]} \wedge \theta_n}} v(X_{\tau_{[a,b]}\wedge\theta_n}) \right] = &
v(x) -
\mathbb E_x \left[ \int_0^{\tau_{[a,b]}\wedge\theta_n} e^{-\rho_s}g ( X_s ) ds \right].
\end{align}
For every stopping time $\theta \leq \tau_{[a,b]}$, we have
\begin{align*}
0 \leq &
e^{- \rho_\theta} = 
1 + \int_0^\theta - e^{-\rho_s} r(X_s) ds =
1 + \int_0^\theta e^{-\rho_s} \left( r^-(X_s) - r^+(X_s) \right) ds \leq
1 + \int_0^{\tau_{[a,b]}} e^{-\rho_s} r^-(X_s) ds .
\end{align*}
Using the Lemmata \ref{L existence of positive solution} and \ref{L integrability 1}, we see that $\mathbb E_x \left[\int_0^{\tau_{[a,b]}} e^{-\rho_s} r^-(X_s) ds \right] < +\infty$.
Since $v$ is bounded in $[a,b]$ and $v(X_{\tau_{[a,b]}})=0$, the Lebesgue dominated convergence theorem states that
\begin{equation}
\notag
\lim_{n \to \infty} \mathbb E_x \left[ e^{-\rho_{\tau_{[a,b]} \wedge \theta_n}} v(X_{\tau_{[a,b]}\wedge\theta_n}) \right] = 0.
\end{equation}
Using the Lebesgue dominated convergence theorem on the right-hand side of \eqref{Z01}, we obtain the Lemma in the case $g \geq 0$. In the general case $g:[a,b] \mapsto \mathbb R$, the Lemma holds for the positive function $|g|$.
Hence, we can apply the Lebesgue dominated convergence theorem to both sides of \eqref{Z01} to finish the proof.
\end{proof}

The following Lemma, together with the preceding ones, allows us to obtain Lemma \ref{L nonintegrability}, from which the Proposition \ref{P sign phi_12} follows.

\begin{lemma}
\label{L probability of right exit}
For every $m < a <b < M$ and every $x \in ]a,b[$:
\begin{equation}
\notag
P_x \left\{ \tau_b < \tau_a \right\} = \frac{\int_a^x e^{-\int_a^{z_1} \frac{2 \alpha}{\sigma^2} dz_2} dz_1}{\int_a^b e^{-\int_a^{z_1} \frac{2 \alpha}{\sigma^2} dz_2} dz_1} .
\end{equation}
In particular, $0 < P_x \left\{ \tau_b < \tau_a \right\} < 1$ for every $x \in ]a,b[$.
\end{lemma}

\begin{proof}
It can be checked that the unique solution of the boundary problem
\[
\alpha v^\prime + \frac{\sigma^2}{2}v^{\prime \prime} =0,
\qquad v(a) = 0, \quad v(b)=1 .
\]
is the function
\[
v(x) = \frac{\int_a^x e^{-\int_a^{z_1} \frac{2 \alpha}{\sigma^2} dz_2} dz_1}{\int_a^b e^{-\int_a^{z_1} \frac{2 \alpha}{\sigma^2} dz_2} dz_1} .
\]
Let $\tau_{[a,b]} = \tau_a \wedge \tau_b = \inf \left\{ t \geq 0: X_t \notin ]a,b[ \right\}$, and let $\{\theta_n\}_{n\in\mathbb{N}}$ be a sequence of stopping times such that $\theta_n\to \tau_I$ and the stopped process $\{X_{t\wedge\theta_n}\}_{t\geq 0}$ is a semimartingale.
By the argument used in the proof of Lemma \ref{L integrability 1},
\begin{align*}
v(X_{\tau_{[a,b]}\wedge\theta_n}) = &
v(x)+
\int_0^{\tau_{[a,b]}\wedge\theta_n}\left( \sigma v'\right) \circ X_s dW_s
\end{align*}
Therefore,
$\mathbb E_x \left[ v \left( X_{\tau_{[a,b]}\wedge \theta_n} \right) \right] =
v(x) $ for every $x \in ]a,b[$.
\\
Since $\tau_{[a,b]}\wedge \theta_n $ converges to $\tau_{[a,b]}$ and $v$ is bounded, the Lebesgue dominated convergence theorem states that
\begin{align*}
&
P_x \left\{ \tau_b < \tau_a \right\} =
\mathbb E_x \left[ v ( X_{\tau_{[a,b]}} ) \right] =
v(x) .
\end{align*}
Due to Assumption \ref{Assumption Engelbert}, for every $x \in ]a,b[$,
\[
0 < 
\int_a^x e^{-\int_a^{z_1} \frac{2 \alpha}{\sigma^2} dz_2} dz_1 <
\int_a^b e^{-\int_a^{z_1} \frac{2 \alpha}{\sigma^2} dz_2} dz_1 <
+\infty .
\]
Therefore, $0 < v(x) < 1$ for every $x \in ]a,b[$.
\end{proof}

The following Lemma concludes the proof of Proposition \ref{P sign phi_12}

\begin{lemma}
\label{L nonintegrability}
Fix $a,b \in I $ such that $a<b$ and $\phi_{12}(a,b) = 0$, then
\begin{equation}
\notag
\mathbb E_x \left[ \int_0^{\tau_{[a^\prime , b^\prime]}} e^{-\rho_s} g(X_s) ds \right] = + \infty
\end{equation}
for every $a^\prime \in ]m,a]$, $b^\prime \in [b,M[$, $x \in ]a^\prime , b^\prime[$, and every measurable function $g \geq 0$, such that $\left\{ x \in [a^\prime,b^\prime ]: g(x)>0 \right\}$ has positive Lebesgue measure.
\end{lemma}

\begin{proof}
Fix $[a,b]$ as above.
Without loss of generality, we may assume that $\phi_{12}(a,x) >0$ for every $x \in ]a,b[$ (take a subinterval if necessary).
Fix $a^\prime \in ]m,a]$, $b^\prime \in [b,M[$, and a measurable function $g \geq 0$ such that $\left\{ x \in [a^\prime,b^\prime ]: g(x)>0 \right\}$ has positive Lebesgue measure.
Due to Lemma \ref{L roots phi_12}, we may assume that $\left\{ x \in [a ,b ]: g(x)>0 \right\}$ has positive Lebesgue measure (shift the interval, if necessary).
\\
For every constant $\varepsilon \in ]0,b-a[$, we have $\phi_{12}(a,x) > 0 $ for every $x \in ]a,b-\varepsilon[$.
By equality \eqref{Eq v nonhomogeneous}, 
\begin{align*}
v_\varepsilon (x) = \frac{\int_a^{b-\varepsilon} \frac{2g(z)}{\sigma(z)^2} \phi_{12}(z,b-\varepsilon) dz}{\phi_{12}(a,b-\varepsilon)} \phi_{12}(a,x) -
\int_a^{x} \frac{2g(z)}{\sigma(z)^2} \phi_{12}(z,x) dz 
\end{align*}
is the unique solution of \eqref{Eq ODE g} with boundary conditions $v(a) = v(b-\varepsilon ) = 0$.
By the Lemma \ref{L integrability 2}, for every $x \in ]a,b[$, we have
\begin{align*}
\mathbb E_x \left[ 
\int_0^{\tau_{[a^\prime, b^\prime]}}
e^{-\rho_s}g(X_s)ds
\right] \geq
\mathbb E_x \left[ 
\int_0^{\tau_{[a, b- \varepsilon ]}}
e^{-\rho_s}g(X_s)ds
\right] =
v_\varepsilon (x) ,
\end{align*}
for every $\varepsilon >0$. 
Since $\lim\limits_{\varepsilon \to 0^+} v_\varepsilon (x) = + \infty$ for every $x \in ]a,b[$, this implies $\mathbb E_x \left[ 
\int_0^{\tau_{[a^\prime, b^\prime]}}
e^{-\rho_s}g(X_s)ds
\right] = + \infty$ for every $x \in ]a,b[$.

Now, fix $c \in ]a,b[$ and $x \in ]a^\prime, b^\prime [ \setminus ]a,b[$.
Assume that $x \in ]a^\prime , c[ $ (the case $x \in ]c,b^\prime[$ is analogous).
Then,
\begin{align*}
\mathbb E_x \left[ 
\int_0^{\tau_{[a^\prime, b^\prime]}}
e^{-\rho_s}g(X_s)ds
\right] \geq & 
\mathbb E_x \left[ 
\int_{\tau_c}^{\tau_{[a^\prime, b^\prime]}}
e^{-\rho_s}g(X_s)ds \chi_{\{ \tau_c < \tau_{a^\prime} \}}
\right] 
\\ = &
\mathbb E_x \left[ 
\int_{\tau_c}^{\tau_{[a^\prime, b^\prime]}}
e^{-(\rho_s-\rho_{\tau_c})}g(X_s)ds \,\, e^{-\rho_{\tau_c}}\chi_{\{ \tau_c < \tau_{a^\prime} \}}
\right] 
\\= &
\mathbb E_c \left[ 
\int_0^{\tau_{[a^\prime, b^\prime]}}
e^{-\rho_s }g(X_s)ds 
\right]
\mathbb E_x \left[ 
 e^{-\rho_{\tau_c}}\chi_{\{ \tau_c < \tau_{a^\prime} \}}
\right].
\end{align*}
By the Lemma \ref{L probability of right exit}, $\mathbb E_x \left[ 
 e^{-\rho_{\tau_c}}\chi_{\{ \tau_c < \tau_{a^\prime} \}}
\right] >0$ 
and therefore the right-hand side of the inequality above is equal to $+\infty$.
\end{proof}

Concerning the proof of Proposition \ref{P assumption Pi vs Pia}, notice that the final argument in the proof of Lemma \ref{L nonintegrability} shows that existence of some $x \in I$ such that $\mathbb E_x \left[ \int_0^{\tau_I} e^{-\rho_t} \Pi^+(X_t ) dt \right] = \infty$ implies that 
\linebreak
$\mathbb E_x \left[ \int_0^{\tau_I} e^{-\rho_t} \Pi^+(X_t ) dt \right] = \infty$ for every $x \in I$.

\subsection{Proof of Proposition \ref{P maximal interval}}\label{SS P maximal interval}

The following Lemma is an easy consequence of Proposition \ref{P sign phi_12}.

\begin{lemma}
\label{L local positive curve}
For any point $x_0 \in I$ such that $v_{ x_0 ,0}(x) <0$ for some $x \in I$, there is a compact interval $[a,b] \subset I$ satisfying \eqref{Eq local positive curve} such that $x_0 \in ]a,b[$.
Conversely, if $[a,b] \subset I$ satisfies \eqref{Eq local positive curve} and there is some $x \in I$ such that $v^{[a,b]}(x) <0$, then there is a compact interval $[a^\prime, b^\prime] \subset I$ satisfying \eqref{Eq local positive curve} such that $[a,b] \subset ]a^\prime , b^\prime[ $.
\end{lemma}

\begin{proof}
Due to Proposition \ref{P sign phi_12}, equality \eqref{Eq v_ad} implies that the mapping $d \mapsto v_{x_0,d}(x_1)$ is strictly increasing for fixed $x_0<x_1$, and strictly decreasing for fixed $x_1 < x_0$.

Fix $x_0, x_1 \in I$ such that $v_{ x_0 ,0}(x_1) <0$, with $x_0 < x_1$ (the case $x_1 < x_0$ is analogous).
Fix $d>0$ sufficiently small such that $v_{x_0,d}(x_1) <0$.
Since $d>0$, there is some $\varepsilon>0$ such that $v_{x_0,d}(x) >0$ for every $x \in ]x_0, x_0+ \varepsilon]$ and $v_{x_0,d}(x) <0$ for every $x \in [x_0- \varepsilon, x_0[$.
Set $b= \min \left\{ x >x_0 : v_{x_0,d}(x) \leq 0 \right\}$.
It is clear that $b \in ] x_0, x_1[$.
Then, there is some $d_1 < v^\prime_{x_0,d}(b)$, such that $v_{b,d_1}(x_0- \varepsilon) <0$.
Let $a = \max \left\{ x \leq x_0 : v_{b,d_1}(x) \leq 0 \right\}$. 
Since $v_{b,d_1}(x) > v_{x_0,d}(x) $ for every $x < b$, it follows that $a \in ]x_0- \varepsilon, x_0[$.
Thus, $x_0 \in ]a,b[$ and $]a,b[$ satisfies \eqref{Eq local positive curve}.

If there is some $x \in ]b,M[$ such that $v^{[a,b]}(x)<0$, then, we can use the argument above taking $v_{a,d}$ with $d > \left( v^{[a,b]} \right)^\prime (a) $.
If there is some $x \in ]m,a[$ such that $v^{[a,b]}(x)<0$, then, we can take $v_{b,d}$ with $d < \left( v^{[a,b]} \right)^\prime (b) $.
\end{proof}

The argument used to prove the Lemma \ref{L local positive curve} can be adapted to prove the following Lemma.

\begin{lemma}
\label{L fusion of intervals}
For any compact intervals $[a,b], \ [a^\prime, b^\prime] \subset I $ satisfying condition \eqref{Eq local positive curve}, such that $a < a^\prime < b < b^\prime$,
\begin{equation}
\notag
v^{[a, b^\prime]} (x) > \max 
\left( v^{[a,b]}(x), v^{[a^\prime,b^\prime]} (x) \right)
\qquad
\forall x \in ]a, b^\prime[ .
\end{equation}
Hence, $[a, b^\prime]$ satisfies \eqref{Eq local positive curve}.
\end{lemma}

\begin{proof}
Let
\begin{equation}
\notag
\hat d = \max \left\{ d \geq 0: v_{a,d}(x) = v^{[a^\prime,b^\prime]}(x) \text{ for some } x \in [a^\prime,b^\prime] \right\} .
\end{equation}
Notice that $\hat d > \left( v^{[a,b]} \right)^\prime (a)$, and therefore $v_{a,\hat d}(x) > v^{[a,b]}(x)$ for every $x>a$.
\\
By continuity, there is some $\hat x \in [a^\prime, b^\prime]$ such that $v_{a, \hat d}(\hat x) = v^{[a^\prime , b^\prime]}( \hat x)$.
If $\hat x \in ]a^\prime , b^\prime[$, then the maximality of $\hat d$ implies that $v_{a,\hat d}^\prime (\hat x) = \left( v^{[a^\prime, b^\prime]} \right)^\prime (\hat x)$.
Thus, by uniqueness of the solution of the ODE \eqref{Eq ODE} with given initial value and derivative, $v_{a,\hat d} = v^{[a^\prime, b^\prime]}$.
Since this is a contradiction, we conclude that $\hat x = b^\prime$ and $v^\prime_{a,\hat d} (b^\prime) < \left( v^{[a^\prime, b^\prime]} \right)^\prime (b^\prime)$. 
Therefore $v^{[a,b^\prime]} = v_{a, \hat d}$ and $ v_{a, \hat d}(x) > v^{[a^\prime, b^\prime]}(x)$ for every $x < b^\prime$.
\end{proof}

The Proposition \ref{P maximal interval} follows from the Lemmata above.

The Lemma \ref{L fusion of intervals} shows that if $\hat x$ lies in some interval satisfying \eqref{Eq local positive curve}, then the union of all intervals containing $\hat x$ and satisfying \eqref{Eq local positive curve} is a maximal interval for \eqref{Eq local positive curve}.
The fact that maximal intervals are pairwise disjoint is also an immediate consequence of Lemma \ref{L fusion of intervals}.

Fix $\hat x \in \mathcal L^+$.
Then, $v^\prime _{\hat x, 0} (x) = - \int_{\hat x}^x \frac{2 \Pi (z)}{\sigma(z)^2} \phi_{22}(z) dz < 0$ for every $x > \hat x$, sufficiently close to $\hat x$.
Therefore, $v_{\hat x, 0} (x) < 0$ for every $x > \hat x$, sufficiently close to $\hat x$, and Lemma \ref{L local positive curve} shows that $\hat x$ lies in some interval satisfying \eqref{Eq local positive curve}.
Conversely, if $[a,b] \subset I$ and $v^{[a,b]}(x) >0$ for every $x \in ]a,b[$, then the equality \eqref{Eq v nonhomogeneous} implies that $\int_x^b \frac{\Pi(z)}{\sigma (z)^2} \phi_{12}(z,b) dz >0$ for some $x \in [a,b[$.
Due to Proposition \ref{P sign phi_12}, this implies $]a,b[ \cap \mathcal L^+ \neq \emptyset $.

If $]a,b[ \subset I$ is maximal for \eqref{Eq local positive curve} then the Lemma \ref{L local positive curve} states that $v^{[a,b]}(x) \geq 0$ for every $x \in I$.
Conversely, any $[a,b] \subset I$ such that $v^{[a,b]}(x) \geq 0$ for every $x \in I$ must be maximal, since any non-negative $v^{[a^\prime, b^\prime]}$, with $a^\prime \leq a$ and $b^\prime \geq b$, must coincide with $v^{[a,b]}$ in at least two points and therefore, by Lemma \ref{L no double cross}, it must coincide with $v^{[a,b]}$.

It only remains to prove that if $]a,b[$ is maximal and $a=m$ or $b=M$, then $v^{[a,b]}$ is well defined and non-negative. Let $]a,b[$ be maximal for \eqref{Eq local positive curve}. 
For any compact intevals $[a_1,b_1]$, $[a_2,b_2]$ satisfying \eqref{Eq local positive curve}, such that $[a_1,b_1] \subset ]a_2,b_2[ $ and $ [a_2,b_2] \subset ]a,b[$, the Lemma \ref{L no double cross} implies that $v^{[a_1,b_1]}(x) < v^{[a_2,b_2]}(x)$ for every $x \in ]a_1,b_1[$.
Hence, for any monotonically increasing sequence of compact intervals $[a_n,b_n] \subset ]a,b[$ satisfying \eqref{Eq local positive curve}, such that $]a,b[ = \bigcup\limits_{n\in \mathbb N} [a_n,b_n]$, the function $v(x) = \lim\limits_{n \rightarrow \infty} v^{[a_n,b_n]}(x)$ is well defined, it is strictly positive in the interval $]a,b[$ and does not depend on  the particular sequence $[a_n,b_n]$.
Further, $v^{[a_n,b_n]}(x)$ and $\left( v^{[a_n,b_n]} \right)^\prime (x)$ converge uniformly on compact intervals.
Hence, $v$ must be a solution of the equation \eqref{Eq ODE} and $v(x) \geq 0 $ for every $x \in I$.

\subsection{Proof of Theorem \ref{T value function}}\label{SS T value function}

First, we will prove a version of Theorem \ref{T value function} under the stronger assumption:
\begin{assumption}
\label{Assumption integrable}
The functions $\frac{1}{\sigma^2}$, $\frac{\alpha}{\sigma^2}$, and $\frac{\Pi}{\sigma^2}$ are integrable with respect to the Lebesgue measure in $I$,
the sets $\left\{ x \in I: \Pi(x) >0 \right\}$ and $\left\{ x \in I: \Pi(x) <0 \right\}$ have both positive Lebesgue measure, and
\begin{equation}
\notag
\mathbb E_x \left[ \int_0^{\tau_I} e^{-\rho_t} \Pi^+ ( X_t) dt \right] < + \infty
\qquad \forall x \in I .
\end{equation}
\end{assumption}
Notice that, contrary to the local integrability required in Assumptions \ref{Assumption Engelbert}, \ref{Assumption r} and \ref{Assumption Pi}, global integrability implies that, for any interval $]a,b[ \subset I$, $v^{[a,b]}$ is well defined by expression \eqref{Eq vab} and $\lim\limits_{x \rightarrow a} v^{[a,b]} (x) = \lim\limits_{x \rightarrow b} v^{[a,b]} (x) = 0$, even if $a=m$ or $b=M$.
Thus, we can consider the compact interval $[m,M]$ instead of $I$.
Conversely, under the Assumptions \ref{Assumption Engelbert}, \ref{Assumption r} and \ref{Assumption Pi}, Assumption \ref{Assumption integrable} holds if we consider a compact subinterval $[a,b]\subset I$ instead of the whole interval $I$.

Under Assumption \ref{Assumption integrable}, the following verification theorem is quite easy to prove.

\begin{teo}
\label{T verification}
Suppose that Assumption \ref{Assumption integrable} holds, and let $v :[m,M] \mapsto [0,\infty[$ be a differentiable function with absolutely continuous derivative.
If $v$ is a solution of the Hamilton-Jacobi-Bellman equation \eqref{Eq HJB} with boundary conditions $v(m)=v(M) =0$, then $v $ coincides with the value function \eqref{Eq value function} in $I$.
\end{teo}

\begin{proof}
Let $V:I \mapsto \mathbb R$ be the value function.
By the argument used to prove the Lemma \ref{L integrability 1}, there is a sequence of increasing stopping times $\{\theta_n\}_{n\in\mathbb{N}}$, converging to $\tau_I$, such that
\begin{align*}
&
\mathbb E_x \left[ e^{-\rho_{\tau \wedge \theta_n}} v(X_{\tau\wedge\theta_n}) \right] =
v(x)+
\mathbb E_x \left[ \int_0^{\tau\wedge\theta_n} e^{-\rho_s} \left( -rv + \alpha v^\prime + \frac{\sigma^2}{2} v^{\prime \prime} \right) \circ X_s ds \right] =
\\ = &
v(x)-
\mathbb E_x \left[ \int_0^{\tau\wedge\theta_n} e^{-\rho_s} \left( rv - \alpha v^\prime - \frac{\sigma^2}{2} v^{\prime \prime} - \Pi \right) \circ X_s ds \right] -
\mathbb E_x \left[ \int_0^{\tau\wedge\theta_n} e^{-\rho_s} \Pi ( X_s) ds \right] 
\end{align*}
for every $x \in I$.
By assumption, $rv - \alpha v^\prime - \frac{\sigma^2}{2} v^{\prime \prime} - \Pi \geq 0$ and $v \geq 0$.
Hence,
\begin{align*}
0 \leq
\mathbb E_x \left[ e^{-\rho_{\tau \wedge \theta_n}} v(X_{\tau\wedge\theta_n}) \right] \leq
v(x) -
\mathbb E_x \left[ \int_0^{\tau\wedge\theta_n} e^{-\rho_s} \Pi ( X_s) ds \right] .
\end{align*}
Letting $n \to \infty$, the Lebesgue dominated convergence theorem guarantees that
\begin{align*}
\mathbb E_x \left[ \int_0^{\tau} e^{-\rho_s} \Pi ( X_s) ds \right] \leq v(x) .
\end{align*}
Since $\tau$ is arbitrary, this proves that $V \leq v$.

Now, fix $x \in I$ and let $\tau = \inf \{ t\geq 0 : v(X_t)=0 \}$.
If $v(x) >0$, then the Lemma \ref{L integrability 2} states that
\begin{align*}
\mathbb E_x \left[ \int_0^{\tau} e^{-\rho_s} \Pi ( X_s) ds \right] = v(x) ,
\end{align*}
and therefore, $V = v$.
\end{proof}

Under Assumption \ref{Assumption integrable}, the Theorem \ref{T value function} takes the following form:

\begin{teo}
\label{T existence HJB}
If Assumption \ref{Assumption integrable} holds, then the Hamilton-Jacobi-Bellman equation \eqref{Eq HJB} admits a solution with boundary conditions $v(m)=v(M)=0$.
This solution is given by the right-hand side of \eqref{Eq solution HJB}.
\end{teo}

\begin{proof}
Let $\left\{ ]a_k,b_k[, \ k=1,2, \ldots \right\} $ be the collection of all maximal intervals for \eqref{Eq local positive curve}, and let $v:I \mapsto [0,+\infty[$ be the function defined by the right-hand side of \eqref{Eq solution HJB}.
\\
It can be checked that $v$ is continuously differentiable with absolutely continuous first derivative, and $\lim\limits_{x \rightarrow m^+} v (x) = \lim\limits_{x \rightarrow M^-} v (x) = 0$.
For almost every $z \in \bigcup \limits_k ]a_k,b_k[$, $v$ satisfies the differential equation \eqref{Eq ODE}. 
By the Proposition \ref{P maximal interval}, $\mathcal L^+ \subset \bigcup \limits_k ]a_k,b_k[$.
Therefore, for almost every $z \in I \setminus \bigcup \limits_k ]a_k,b_k[$:
\begin{align*}
r(z) v(z)-\alpha(z) v^\prime (z)-\frac{\sigma(z)^2}{2} v^{\prime \prime} (z)-\Pi(z) = - \Pi(z) \geq 0 .
\end{align*}
Hence, $v$ is a solution of the Hamilton-Jacobi-Bellman equation \eqref{Eq HJB}.
\end{proof}

The Theorem \ref{T value function} follows easily from Theorem \ref{T existence HJB}.
To see this, for every compact interval $[a,b] \subset I$, let
\begin{equation}
\notag
V^{[a,b]}(x) = \sup_{\tau \in \mathcal T, \ \tau \leq \tau_{[a,b]}} \mathbb E_x \left[ \int_0^\tau e^{- \rho_t} \Pi(X_t) dt \right] .
\end{equation}
This function is given by Theorem \ref{T existence HJB} with the interval $[a,b]$ replaced by $I$.
\\
For every stopping time $\tau \in \mathcal T$, and every monotonically increasing sequence $[a_n,b_n] \subset I$ such that $I= \bigcup\limits_{n \in \mathbb N} [a_n,b_n]$, the Lebesgue monotone convergence theorem states that
\begin{align*}
&
\lim_{n \rightarrow \infty} \mathbb E_x \left[ \int_0^{\tau \wedge \tau_{[a_n,b_n]}} e^{- \rho_t} \Pi^-(X_t) dt \right] =
\mathbb E_x \left[ \int_0^\tau  e^{- \rho_t} \Pi^-(X_t) dt \right] ,
\\ &
\lim_{n \rightarrow \infty} \mathbb E_x \left[ \int_0^{\tau \wedge \tau_{[a_n,b_n]}} e^{- \rho_t} \Pi^+(X_t) dt \right] =
\mathbb E_x \left[ \int_0^\tau  e^{- \rho_t} \Pi^+(X_t) dt \right] .
\end{align*}
Hence, the value function $V$ satisfies
\begin{equation}
\notag
V(x) = \lim_{n \rightarrow \infty} V^{[a_n,b_n]}(x) \qquad \forall x \in I .
\end{equation}
By Definition \ref{D maximal interval} and Proposition \ref{P maximal interval}, $V$ is given by \eqref{Eq solution HJB}.

\subsection{Proof of Proposition \ref{P V intervals}}\label{SS P V intervals}

Fix $a \in I$, $b \in ]a,M]$, and suppose that $]a,b[$ is maximal for \eqref{Eq local positive curve}.
By Proposition \ref{P maximal interval}, $v^{[a,b]} \geq 0$.
The proof of Proposition \ref{P maximal interval} shows that $v^{[a,b]}$ is a solution of the differential equation \ref{Eq ODE}, even in the case $b=M$.
Hence $v^{[a,b]} = v_{a,0}$ and {\it (a)} holds.
Fix $[a_1,b_1] \subset ]a,b[$, a compact interval satisfying \eqref{Eq local positive curve}.
Then, there is an interval $]a_2,b_1[$, maximal for \eqref{Eq local positive curve} when we consider the interval $]m,b_1[$ instead of $I$.
By the Proposition \ref{P maximal interval}, $v^{[a_2,b_1]}$ must be non-negative in $]m,b_1[$.
Hence, $v^{[a_2,b_1]}= v_{a_2,0}$.
By the considerations preceeding Theorem \ref{T verification}, $v_{a_2,0}(b_1) = 0$.
Since $v_{a_2,0}(x) >0 $ for every $x>a_2$ sufficiently close to $a_2$, it follows that there is some $a_3 \in \mathcal L^-$ arbitrarily close to $a_2$.
Thus, {\it (b)} also holds.

Now, fix $a \in I$, $b \in ]a,M]$, and suppose that {\it (a)} and {\it (b)} hold.
Let $a_n$ be a sequence as in {\it (b)}, and let $b_n= \inf \left\{ x >a_n : v_{a_n,0}(x) \leq 0 \right\}$.
Since $]a,b[ = \bigcup \limits_{n \in \mathbb N} ]a_n,b_n[$, the Lemma \ref{L fusion of intervals} guarantees that $]a,b[$ satisfies \eqref{Eq local positive curve}.
Due to Lemma \ref{L no double cross}, non-negativity of $v_{a,0}$ implies that $]a,b[$ is maximal for \eqref{Eq local positive curve}.

The proof for the case $b \in I$, $a \in [m,b[$ is analogous.

\bibliographystyle{plain}
\bibliography{myrefs}


\end{document}